\documentclass{amsart}

\usepackage[all]{xy}

\usepackage{latexsym}

\usepackage{amssymb}
\usepackage{amsfonts}
\usepackage{amscd}
\usepackage{amsmath,amsthm}
\usepackage{xcolor}

\newtheorem{lemma}{Lemma}[section]
\newtheorem{proposition}[lemma]{Proposition}
\newtheorem{theorem}[lemma]{Theorem}
\newtheorem{corollary}[lemma]{Corollary}

\newtheorem{prop}[lemma]{Proposition}

{

}

\theoremstyle{definition}

\newtheorem{example}[lemma]{Example}
\newtheorem{definition}[lemma]{\sl Definition}

\newtheorem{question}[lemma]{Question}

\theoremstyle{remark}

\newtheorem{remark}[lemma]{Remark}

\newcommand{\Hom}{\operatorname{Hom}}

\newcommand{\cL}{\mathcal{L}}

\newcommand{\rk}{\operatorname{rk}}
\newcommand{\cE}{\mathcal{E}}
\newcommand{\cF}{\mathcal{F}}

\newcommand{\cvec}[2]{\begin{pmatrix} #1 \\ #2 \end{pmatrix}}
\newcommand{\cvecd}{\cvec{d_{-1}}{d_0}}
\newcommand{\cvecdp}{\cvec{d'_{-1}}{d'_0}}
\newcommand{\cvecr}{\cvec{r_{-1}}{r_0}}

\newcommand{\seed}{d_{-1},d_0,r_{-1},r_0}

\newcommand{\seedmatrix}{\begin{vmatrix}
    r_{-1} & d_{-1} \\ r_0 & d_0
\end{vmatrix} }

\numberwithin{equation}{section}

\begin{document}

\title{Two-periodic elliptic helices: classification and geometry}

\author{D. Chan}
\address{University of New South Wales}
\email{danielc@unsw.edu.au}

\author{A. Nyman}
\address{Western Washington University}
\email{adam.nyman@wwu.edu}

\begin{abstract}
Let $k$ denote an algebraically closed field of characteristic zero and let $X$ denote a smooth elliptic curve over $k$.  In this paper, motivated by work in \cite{CN}, we think of two-periodic elliptic helices as noncommutative analogues of degree two line bundles over $X$.  We classify and study two-periodic elliptic helices in order to generalize the theory of double covers of $\mathbb{P}^{1}$ by $X$ to the noncommutative setting.  This leads to the following problem:  given an integer $d>2$ and a real number $\theta \in \mathbb{Q}+\mathbb{Q}\sqrt{d^2-4}$, classify elliptic helices inducing double covers of $\mathbb{P}^{1}_{d}$ by ${\sf C}^{\theta}$, where $\mathbb{P}^{1}_{d}$ is Piontkovski's noncommutative projective line and ${\sf C}^{\theta}$ is Polischuk's noncommutative elliptic curve.  We find examples of $d$ and $\theta$ such that there is essentially one numerical class of elliptic helices and examples of $d$ and $\theta$ such that there are several distinct numerical classes of elliptic helices, in contrast to the commutative situation.  
\end{abstract}

\maketitle

\pagenumbering{arabic}

\section{Introduction}
Throughout this paper, we let $k$ denote an algebraically closed field of characteristic zero and we let $X$ denote a smooth elliptic curve over $k$.  Recall the following result, which classifies linear systems of dimension one and degree two, $g^1_2$'s, on $X$. 
\begin{prop} \label{prop.commutative}
Let $\mathcal{L}$ be a degree two line bundle over $X$.  Then 
\begin{enumerate}
\item{} any degree two line bundle over $X$ is isomorphic to $\mathcal{L}_{0} \otimes \mathcal{L}$, where $\mathcal{L}_{0}$ is a degree zero line bundle.  Furthermore, degree zero line bundles are parameterized by $X$ \cite{Ati},  

\item{} given any choice of generating global sections of $\mathcal{L}$, one obtains a double cover $f:X \rightarrow \mathbb{P}^{1}$ such that $f^{*}\mathcal{O}(1) \cong \mathcal{L}$, and

\item{} given two double covers $f,g:X \rightarrow \mathbb{P}^{1}$, there exist automorphisms $\sigma$ of $X$ and $\tau$ of $\mathbb{P}^{1}$ such that the diagram
$$
\begin{CD}
X & \overset{\sigma}{\longrightarrow} & X \\
@V{f}VV @VV{g}V \\ 
\mathbb{P}^{1} & \underset{\tau}{\longrightarrow} & \mathbb{P}^{1}
\end{CD}
$$
commutes \cite[Chapter IV, Lemma 4.4]{hartshorne}.
\end{enumerate}
\end{prop}
The purpose of this paper is to present a noncommutative generalization of this picture.

Three-periodic elliptic helices in a triangulated category were introduced by Bondal and Polishchuk in \cite{bp}, and three-periodic elliptic helices composed of line bundles over $X$ were used by Bondal and Polishchuk in order to construct and classify Artin-Schelter regular algebras of global dimension and Gorenstein parameter three.  In order to describe the utility of two-periodic elliptic helices initially studied in \cite{CN}, recall that for every $\theta \in \mathbb{R}$, Polishchuk defined a nonstandard $t$-structure on $D^{b}(X)$, and defined the abelian category ${\sf C}^{\theta}$ as the heart of this $t$-structure \cite{polish}.  He then proved that ${\sf C}^{\theta}$ has cohomological dimension one and is derived equivalent to $X$ (\cite[Proposition 1.4]{polish}), justifying the perspective that ${\sf C}^{\theta}$ is a noncommutative elliptic curve.  In addition, for $d \geq 2$, the space $\mathbb{P}^{1}_{d}$ studied by Piontkovski \cite{piont} is a noncommutative analogue of the projective line (see Section \ref{sub.ncp1} for more details).  In \cite{CN}, two-periodic elliptic helices composed of simple vector bundles over $X$ (see Definition \ref{def:mutable} for details) were constructed and used to build double covers of $\mathbb{P}^{1}_{d}$ by ${\sf C}^{\theta}$ for certain values of $d$ and $\theta$.  Based on the constructions in \cite{CN}, it seems reasonable to think of two-periodic elliptic helices as noncommutative analogues of degree two line bundles over $X$, and we pursue that perspective here.  For this reason, in this paper we classify all two-periodic elliptic helices, $\underline{\mathcal{E}} := (\mathcal{E}_{i})_{i \in \mathbb{Z}}$, over $X$ (Theorem \ref{thm.elliptic2extend}):

\begin{theorem} \label{theorem.class}
Suppose $\mathcal{E}_{-1}$ and $\mathcal{E}_{0}$ are simple vector bundles over $X$ with slopes $\mu_{-1}$ and $\mu_{0}$, respectively, such that $\mu_{-1}<\mu_{0}$.  Let $d:= \operatorname{dim }\operatorname{Hom}(\mathcal{E}_{-1}, \mathcal{E}_{0})$.  Then the pair $(\mathcal{E}_{-1},\mathcal{E}_{0})$ extends, via mutation, to an elliptic helix of period two if and only if
\begin{itemize}
\item{} $d=2$.  In this case, $r_{-1}=r_{0}=1$ and $d_{0}=d_{-1}+2$, or

\item{} $d>2$ and
\begin{equation} \label{eqn.helixcondD}
dr_{-1}r_{0}-r_{-1}^{2}-r_{0}^{2}>0.
\end{equation}
\end{itemize}
\end{theorem}
The numerical condition \cite[Lemma 7.6]{morph}, which we didn't know when \cite{CN} was written, plays a key role in the proof.  We next show that, in case $d>2$, the sequence of slopes (in the negative direction) of the vector bundles in a two-periodic elliptic helix has an irrational limit, $\theta$ (Lemma \ref{lemma.limit}).  The value of $\theta$, which is always in $\mathbb{Q} + \mathbb{Q}\sqrt{d^2-4}$, is explicitly given in terms of the numerical data defining $\mathcal{E}_{-1}$ and $\mathcal{E}_{0}$.  We also prove that all two-periodic elliptic helices are ample (Proposition \ref{prop.smallample}).  The significance of this last result is that the $\mathbb{Z}$-algebra $B_{\underline{\mathcal{E}}}$ with $i,j$-component $\operatorname{Hom}(\mathcal{E}_{-j}, \mathcal{E}_{-i})$ and multiplication induced by composition is a noncommutative homogeneous coordinate ring of the noncommutative elliptic curve ${\sf C}^{\theta}$.  As $\theta$ is irrational, $B_{\underline{\mathcal{E}}}$ is coherent and nonnoetherian (\cite[Proposition 3.1]{polish}).

Before we describe our noncommutative version of Proposition \ref{prop.commutative}(1), we relate the $d=2$ case of Theorem \ref{theorem.class} to the commutative case.  The following terminology will be useful:

\begin{definition} \label{def:numericalShift}
Let $\underline{\mathcal{E}}$ be a two-periodic elliptic helix.  If $\mathcal{L}$ denotes a line bundle over $X$, we let $\mathcal{L} \otimes \underline{\mathcal{E}}$ denote the two-periodic elliptic helix with $i$th component $\mathcal{L} \otimes \mathcal{E}_{i}$ and call this the {\it twist} of $\underline{\mathcal{E}}$ by $\mathcal{L}$.  If $m \in \mathbb{Z}$, we let $\underline{\mathcal{E}}[m]$ denote the two-periodic elliptic helix with $i$th component $\mathcal{E}_{i+m}$ and call $\underline{\mathcal{E}}[m]$ the {\it shift of $\underline{\mathcal{E}}$ by $m$}.  We say $\underline{\mathcal{E}}$ and $\underline{\mathcal{F}}$ are {\it numerically equal} if $\operatorname{rank }\mathcal{F}_{i} = \operatorname{rank }\mathcal{E}_{i}$ and $\operatorname{deg }\mathcal{F}_{i}=\operatorname{deg }\mathcal{E}_{i}$ for all $i$.  We say two-periodic elliptic helices are {\it in the same numerical class} if they are numerically equal after shifting and twisting.
\end{definition}
Now suppose $\underline{\mathcal{E}}$ is a two-periodic elliptic helix with $d=2$. Then, by Theorem \ref{theorem.class} (and Remark \ref{remark.helix}(3)), we may twist $\underline{\mathcal{E}}$ by either a degree zero or degree one line bundle so that some component of our helix is $\mathcal{O}_{X}$ while the next is a degree two line bundle $\mathcal{L}$.   It follows that, up to shifting, we may assume $\mathcal{E}_{-1} \cong \mathcal{O}_{X}$ and $\mathcal{E}_{0} \cong \mathcal{L}$.  Since a two-periodic elliptic helix is determined by any two subsequent terms (see Remark \ref{remark.helix}(4)), the terms of $\underline{\mathcal{E}}$ must be isomorphic to tensor powers of $\mathcal{L}$ by \cite[Example 3.5]{CN}.  Notice that, in this case, the negative limit slope $\theta$ is $-\infty$, and every pair of two-periodic elliptic helices are in the same numerical class.   

This leads us to the following 

\begin{question} \label{question.unique}
For which values of $d>2$ and $\theta$ is there a unique numerical class of two-periodic elliptic helices?  If there is not a unique numerical class, how many classes are there?
\end{question}
 
We address Question \ref{question.unique} in Section \ref{section.diffcovers} for various values of $d$ and $\theta$, but a comprehensive solution is beyond our reach.  For example, we show (Proposition \ref{prop.caseyis1}) that if $\theta \in \mathbb{Q}-\frac{d}{2(d-2)}\sqrt{d^2-4}$, then any pair of two-periodic elliptic helices are in the same numerical class.  Although twisting will change the rational part of the negative limit slope by the degree of the twisting bundle, the value of $d$ is unchanged (see Section \ref{section.helixops}).  Therefore, any two elliptic helices with equal $\theta$ (in the set above) and equal $d>2$ are numerically equal up to shifting.  This holds, in particular, for the values of $d$ and $\theta$ studied in \cite{CN}.  

We also find many examples for which the answer to Question \ref{question.unique} is negative.  For example, we prove (Proposition \ref{prop.caseyis3}) that if $d=10$ and $\theta \in \mathbb{Z}-\sqrt{6}$, then there are two distinct numerical classes of two-periodic elliptic helices.  

The heart of the investigation of Question \ref{question.unique} involves the classification of solutions to an equation related to a Pell-like equation, as we describe below.  Given a two-periodic elliptic helix $\underline{\mathcal{E}}$, the left-hand side of (\ref{eqn.helixcondD}) is an associated positive integer $D$ such that every consecutive pair of ranks, $r_{i-1}, r_{i}$ of $\mathcal{E}_{i-1}, \mathcal{E}_{i}$ satisfies the equation
\begin{equation} \label{eqn.orig}
dr_{i-1}r_{i}-r_{i-1}^{2}-r_{i}^{2} = D.
\end{equation}
Diagonalizing the corresponding quadratic form yields the quadratic form
\begin{equation} \label{eqn.formintro}
(d-2)x^2-(d+2)y^{2}=4D
\end{equation}
which has the property that $(r_{i-1}+r_{i}, r_{i}-r_{i-1})$ is an integer solution for all $i$. Furthermore, any other two-periodic elliptic helix with the same limit slope and the same value of $d$ also provides solutions to (\ref{eqn.formintro}).  By an observation of Lagrange, there is a Pell-like equation whose integer solutions are related to solutions to (\ref{eqn.formintro}) \cite[Section 17]{pell}.  Specifically, for every integer solution $(x,y)$ to (\ref{eqn.formintro}), the pair $(X,Y)$ with $X=4(d^2-4)y$ and $Y=2(d-2)x$ is an integer solution to
\begin{equation} \label{eqn.pellpell}
X^{2}-4(d^{2}-4)Y^2=-64(d^{2}-4)(d-2)D.
\end{equation}
In general, there is a finite set $S$ of solutions to (\ref{eqn.pellpell}) such that every solution to (\ref{eqn.pellpell}) is a so-called Pell associate of a solution in $S$.  See \cite[Section 17]{pell} for more details.  Our analysis of solutions to (\ref{eqn.orig}), however, doesn't utilize the theory of Pell-like equations.

We next turn to the study of morphisms induced by two-periodic elliptic helices (Section \ref{sec.doubleCovers}), and note that a choice of basis of $\operatorname{Hom}(\mathcal{E}_{-1}, \mathcal{E}_{0})$ induces a double cover (i.e. a functor)
$$
F_{\underline{\mathcal{E}}}:{\sf C}^{\theta} \longrightarrow \mathbb{P}^{1}_{d}
$$
which has a left-adjoint $F_{\underline{\mathcal{E}}}^{*}$ (see (\ref{eqn.maptop1})).  We prove that if $\mathcal{O}(i)$ is the twist of the canonical structure sheaf of $\mathbb{P}^{1}_{d}$, then $F_{\underline{\mathcal{E}}}^{*}(\mathcal{O}(i)) \cong \mathcal{E}_{i}$ for all $i$ (Proposition \ref{prop:pullbackFunctor}).  This is our noncommutative version of Proposition \ref{prop.commutative}(2).  It is interesting to note that, as was mentioned above, the values of $\theta$ that arise are constrained to be in $\mathbb{Q}+\mathbb{Q}\sqrt{d^2-4}$.  Noncommutative elliptic curves derived equivalent to $X$ are parameterized by the real numbers, and it is unclear whether it is possible to extend our results to other values of $\theta$.

We use Proposition \ref{prop:pullbackFunctor} to prove the following noncommutative generalization of Proposition \ref{prop.commutative}(3)  (Proposition \ref{prop.twistTranslateHelix}):

\begin{proposition} 
Let $\underline{\mathcal{E}}, \underline{\mathcal{F}}$ be elliptic helices such that $\deg \mathcal{E}_i = \deg \mathcal{F}_i = d_i$,  $\rk \mathcal{E}_i = \rk \mathcal{F}_i=  r_i$ for $i = -1, 0$. Then there exist auto-equivalences $\sigma$ of ${\sf C}^{\theta}$ and $\tau$ of $\mathbb{P}_{d}^{1}$ such that the diagram of functors
$$
\begin{CD}
{\sf C}^{\theta} & \overset{\sigma}{\longrightarrow} & {\sf C}^{\theta} \\
@V{F_{\underline{\mathcal{E}}}}VV @VV{F_{\underline{\mathcal{F}}}}V \\ 
\mathbb{P}^{1}_{d} & \underset{\tau}{\longrightarrow} & \mathbb{P}^{1}_{d}
\end{CD}
$$
commutes up to natural equivalence.
\end{proposition}
Given this result, one might ask what the group of autoequivalences of $\mathbb{P}^{1}_{d}$ and ${\sf C}^{\theta}$ are.  We determine the autoequivalences of $\mathbb{P}^{1}_{d}$ in Proposition \ref{prop.autoP1d}:

\begin{proposition}
The group of auto-equivalences of $\mathbb{P}^{1}_{d}$ is isomorphic to $\mathbb{Z} \times PGL_d$. The factor $\mathbb{Z}$ is generated by graded shifts whilst $PGL_d$ corresponds to automorphisms of $k^d$.
\end{proposition}
Our main tool for the proof of this proposition and Proposition \ref{prop:pullbackFunctor} is a $\mathbb{Z}$-algebra generalization of the Eilenberg-Watts theorem \cite[Theorem 1.2]{ewsmith} (see Section \ref{section.ew}).

Although the autoequivalences of $D^{b}(X)$ were computed in \cite{polish}, we have not been able to determine the autoequivalences of ${\sf C}^{\theta}$, but the problem seems an interesting one (see Remark \ref{rem.ellipticaut}).  
\medskip
\begin{flushleft}
{\it Notation and conventions:}   We will let ${\sf Coh }X$ denote the category of coherent sheaves over $X$.  The symbol $\mu$ (decorated by a subscript) will denote the slope of a vector bundle, that is, the degree of the bundle divided by its rank.  Moreover 
\medskip
\begin{center} {\it ${\sf C}$ will denote a $k$-linear, $\operatorname{Hom}$-finite abelian category,}
\end{center}
\medskip
where $\operatorname{Hom}$-finite means that the $k$-dimension of all $\operatorname{Hom}$-spaces over ${\sf C}$ are finite dimensional.  Other notation will be introduced locally.
\end{flushleft}

\section{Definition and basic properties of two-periodic elliptic helices}

In this section, we recall, from \cite{CN} and \cite{morph},  the definition and basic properties of two-periodic elliptic helices over the category ${\sf Coh }X$.  The definitions are motivated by \cite[Section~7]{bp}.

\begin{definition}
An object $\cE \in {\sf C}$ is {\it elliptically exceptional} if i) $\operatorname{Ext}^j(\cE,\cE) \cong k$ for $j = 0,1$ and is zero otherwise and ii) for any $\cF \in {\sf C}$, the natural pairing $\operatorname{Hom}(\cE,\cF) \otimes \,\operatorname{Ext}^1(\cF,\cE) \to \,\operatorname{Ext}^1(\cE,\cE) = k$ is non-degenerate.
\end{definition}
It follows (see \cite[Section 6]{morph}) that a coherent sheaf $\cE$ over $X$ is elliptically exceptional in ${\sf Coh }X$ if and only if  $\cE$ is either a vector bundle with relatively prime degree and rank or $\cE$ is a skyscraper sheaf.

We recall from \cite{az2} that if $\mathcal{E}$ and $\mathcal{F}$ are objects of ${\sf C}$ such that $\operatorname{End }\mathcal{E} = k = \operatorname{End }\mathcal{F}$, one can define objects $\operatorname{Hom}(\mathcal{E}, \mathcal{F}) \otimes_{k} \mathcal{E}$ and ${}^{*}\operatorname{Hom}(\mathcal{E}, \mathcal{F}) \otimes_{k} \mathcal{F}$ of ${\sf C}$ as in \cite[Section B3]{az2}.  Moreover, given a choice of basis for $\operatorname{Hom}(\mathcal{E}, \mathcal{F})$, there are canonical maps, the evaluation map
$$
\epsilon: \operatorname{Hom}(\mathcal{E}, \mathcal{F}) \otimes_{k} \mathcal{E} \rightarrow \mathcal{F},
$$
and the coevaluation map
$$
\eta: \mathcal{E} \rightarrow {}^{*}\operatorname{Hom}(\mathcal{E}, \mathcal{F}) \otimes_{k} \mathcal{F}.
$$

\begin{definition}  \label{def:mutable}
An ordered pair $\cE,\cF$ of objects in ${\sf C}$ is said to be {\it left mutable} if $\operatorname{Ext}^j(\cE,\cF) = 0$ for $j \neq 0$ and furthermore the evaluation map
$$
\epsilon \colon \operatorname{Hom}(\cE,\cF) \otimes \cE \to \cF
$$
is surjective. In this case we define the {\it left mutation} $L_\cE \cF := \ker \epsilon$, and remark that, up to isomorphism, it does not depend on choice of basis of $\operatorname{Hom}(\mathcal{E}, \mathcal{F})$. We also say {\it $\cF$ left mutates through $\cE$ in ${\sf C}$}. The right-handed versions are defined similarly using coevaluation.
\end{definition}

\begin{definition}  \label{def:ellipticHelix}
Let $\underline{\cE} = (\cE_i)_{i \in \mathbb{Z}}$ be a sequence of elliptically exceptional objects in ${\sf Coh }X$. We say $\underline{\cE}$ is a {\it two-periodic elliptic helix } if
\begin{enumerate}
    \item for all $i<j$ we have $\operatorname{Ext}^1(\cE_i,\cE_j) = 0$ and,
    \item for all $i$, $\cE_i$ left mutates through $\cE_{i-1}$, and this left mutation is isomorphic to $\cE_{i-2}$.
\end{enumerate}
\end{definition}

\begin{remark} \label{remark.helix}
\begin{enumerate}
\item{} By \cite[Lemma 3]{rudakov}, a two-periodic elliptic helix can also be defined using right mutation.

\item{} By the remark following \cite[Question 7.1]{morph}, if $\underline{\cE}$ is a helix, then $\mathcal{E}_{i}$ is a simple bundle for all $i \in \mathbb{Z}$.

\item{} Every two-periodic elliptic helix $\underline{\cE}$ defines, by mutation, a helix in the sense of \cite[Definition 3.1]{CN} since $\operatorname{Hom}(\mathcal{E}_{i},\mathcal{E}_{i})=k$.  In particular, by \cite[Lemma 3]{rudakov}, for each $i$ there is a short exact sequence 
\begin{equation} \label{eqn.eulerelliptic}
0 \longrightarrow \mathcal{E}_{i} \longrightarrow \operatorname{Hom}(\mathcal{E}_{i+1}, \mathcal{E}_{i+2}) \otimes \mathcal{E}_{i+1} \overset{\epsilon}{\longrightarrow} \mathcal{E}_{i+2} \longrightarrow 0, 
\end{equation}
and ${}^{*}\operatorname{Hom}(\mathcal{E}_{i}, \mathcal{E}_{i+1}) \cong \operatorname{Hom}(\mathcal{E}_{i+1}, \mathcal{E}_{i+2})$.  

\item{} It follows from (\ref{eqn.eulerelliptic}) that every two-periodic elliptic helix is determined up to term-wise isomorphism by any consecutive pair of terms.  We will often use this fact in the sequel.
\end{enumerate}
\end{remark}

\section{Classification of two-periodic elliptic helices}
In this section, we determine necessary and sufficient conditions for a pair of simple vector bundles over $X$, $\mathcal{E}_{-1}$, $\mathcal{E}_{0}$, with ranks $r_{-1}$ and $r_{0}$,  degrees $d_{-1}$ and $d_{0}$, and slopes $\mu_{-1}<\mu_{0}$, to extend to a two-periodic elliptic helix. 

\subsection{Preliminaries and notation} \label{sub.notation}
Recall that, by \cite[Lemma 7.6]{morph} applied to both $(\mathcal{E}_{-1}, \mathcal{E}_{0})$ and $(\mathcal{E}_{0}^{*}, \mathcal{E}_{-1}^{*})$, if 
$$
d:=\operatorname{dim} \operatorname{Hom}(\mathcal{E}_{-1},\mathcal{E}_{0})>0
$$ 
and we let $r_{i}$, $i \in \mathbb{Z}$, be defined by 
$$
\begin{pmatrix} r_{n-1} \\ r_{n} \end{pmatrix} := \begin{pmatrix} 0 & 1 \\ -1 & d \end{pmatrix}^{n} \begin{pmatrix} r_{-1} \\ r_{0} \end{pmatrix}
$$
(see (\ref{eqn.eulerelliptic})), then the pair $\mathcal{E}_{-1}$, $\mathcal{E}_{0}$ extends to a two-periodic elliptic helix if and only if $r_{i}>0$ for all $i$.  In order to find necessary and sufficient conditions under which this occurs, we first describe the relevant linear algebra in a slightly more general context.

Suppose $d \in \mathbb{R}$ is such that $|d| \geq 2$, and $A=\begin{pmatrix} 0 & 1 \\ -1 & d \end{pmatrix}$.  Then $A^{-1}=\begin{pmatrix} d & -1 \\ 1 & 0 \end{pmatrix}$, and both $A$ and $A^{-1}$ have characteristic equation
\begin{equation} \label{eqn.char}
cp_A(\lambda) = \lambda^2-d\lambda+1=0
\end{equation}
whose (real) roots are $\alpha_{\pm}=\frac{d}{2}\pm \frac{\sqrt{d^2-4}}{2}$.  The corresponding eigenvectors of both $A$ and $A^{-1}$ are $\begin{pmatrix} 1 \\ \alpha_{+} \end{pmatrix}$ and $\begin{pmatrix} 1 \\ \alpha_{-} \end{pmatrix}$.  Motivated by (\ref{eqn.eulerelliptic}), given pairs of real numbers $(r_{-1}, d_{-1})$ and $(r_{0},d_{0})$, we define , sequences of real numbers $(r_{i})_{i \in \mathbb{Z}}$ and $(d_{i})_{i \in \mathbb{Z}}$ by the formula
$$
\begin{pmatrix} a_{n-1} \\ a_{n} \end{pmatrix} = A^{n} \begin{pmatrix} a_{-1} \\ a_{0} \end{pmatrix}
$$
where either $a_{i}=r_{i}$ for all $i$ or $a_{i}=d_{i}$ for all $i$.

\begin{itemize}
\item{} If $d=2$, there is only one eigenvalue, $\lambda = 1$, whose corresponding eigenvectors are spanned by $\begin{pmatrix} 1 \\ 1 \end{pmatrix}$.  Thus, we may write $\begin{pmatrix} r_{-1} \\ r_{0} \end{pmatrix}$ in terms of a Jordan basis of $\begin{pmatrix} 0 & 1 \\ -1 & 2 \end{pmatrix}$ as follows:
$$
\begin{pmatrix} r_{-1} \\ r_{0} \end{pmatrix} = r_{0} \begin{pmatrix} 1 \\1  \end{pmatrix} +(r_{0}-r_{-1})\begin{pmatrix} -1 \\ 0 \end{pmatrix}.
$$
In particular, for $n \geq 0$, we have
$$
\begin{pmatrix} r_{n-1} \\ r_{n} \end{pmatrix} = A^n \begin{pmatrix} r_{-1} \\ r_{0} \end{pmatrix}=\begin{pmatrix}
nr_{0}-(n-1)r_{-1} \\ (n+1)r_{0}-nr_{-1}.
\end{pmatrix}
$$
Similarly, we may write $\begin{pmatrix} r_{-1} \\ r_{0} \end{pmatrix}$ in terms of a Jordan basis of $\begin{pmatrix} 2 & -1 \\ 1 & 0 \end{pmatrix}$:
$$
\begin{pmatrix} r_{-1} \\ r_{0} \end{pmatrix} = r_{0} \begin{pmatrix} 1 \\1  \end{pmatrix} +(r_{-1}-r_{0})\begin{pmatrix} 1 \\ 0 \end{pmatrix},
$$
so that, for $n \geq 0$, we have
$$
\begin{pmatrix} r_{-(n+1)} \\ r_{-n} \end{pmatrix} = A^{-n}\begin{pmatrix} r_{-1} \\ r_{0} \end{pmatrix} = \begin{pmatrix}
-nr_{0}+(n+1)r_{-1} \\ (1-n)r_{0}+nr_{-1}.
\end{pmatrix}
$$
In particular, for all $n \in \mathbb{Z}$, we have
\begin{equation} \label{eqn.r}
r_{n}=(n+1)r_{0}-nr_{-1}.
\end{equation}

\item{} If $|d|>2$, there are two distinct one-dimensional eigenspaces and we may write
\begin{equation} \label{eqn.evectorexpand}
\begin{pmatrix} a_{-1} \\ a_{0} \end{pmatrix} = a_{+} \begin{pmatrix} 1 \\ \alpha_{+} \end{pmatrix} + a_{-}\begin{pmatrix} 1 \\ \alpha_{-} \end{pmatrix}
\end{equation}
for $a=r$ or $a=d$.  Thus, for all $n \in \mathbb{Z}$, we have
\begin{equation} \label{eqn.a}
a_{n}=\alpha_{+}^{n+1}a_{+}+\alpha_{-}^{n+1}a_{-},
\end{equation}

\end{itemize}

\subsection{Statement and proof of the classification}
In this section, we retain the notation from Section \ref{sub.notation}.  We have the following classification result.

\begin{theorem} \label{thm.elliptic2extend}
Suppose $\mathcal{E}_{-1}$ and $\mathcal{E}_{0}$ are simple vector bundles over $X$ whose slopes, $\mu_{-1}$ and $\mu_{0}$ satisfy $\mu_{-1}<\mu_{0}$.  Let $d:= \operatorname{dim }\operatorname{Hom}(\mathcal{E}_{-1}, \mathcal{E}_{0})$.  Then the ordered pair $(\mathcal{E}_{-1},\mathcal{E}_{0})$ extends to an elliptic helix of period two if and only if
\begin{itemize}
\item{} $d=2$.  In this case, $r_{-1}=r_{0}=1$ and $d_{0}=d_{-1}+2$, or

\item{} $d>2$ and
\begin{equation} \label{eqn.helixcond}
\alpha_{-} < \frac{r_{-1}}{r_{0}} < \alpha_{+}
\end{equation}
where $\alpha_{\pm}=\frac{d \pm \sqrt{d^{2}-4}}{2}$.

\end{itemize}
\end{theorem}

\begin{proof} Retain the notation from Section \ref{sub.notation}, so that $\mathcal{E}_{-1}$ and $\mathcal{E}_{0}$ have rank and degree $(r_{-1}, d_{-1})$, $(r_{0}, d_{0})$ as in Section \ref{sub.notation}.

If $d=1$ and $(\mathcal{E}_{-1}, \mathcal{E}_{0})$ extended to an elliptic helix, then we would have exact sequences $0 \rightarrow \mathcal{E}_{i-1} \rightarrow \mathcal{E}_{i} \rightarrow \mathcal{E}_{i+1} \rightarrow 0$ for $i \geq 0$ by Remark \ref{remark.helix}.  In particular, $r_{1}=r_{0}-r_{-1}$ while $r_{2}=r_{1}-r_{0}=-r_{-1}<0$, a contradiction.

Now, suppose $d=2$.  In this case, we show that if $(\mathcal{E}_{-1}, \mathcal{E}_{0})$ extends to an elliptic helix, then $r_{i}=1$ and $d_{i}=d_{i-1}+2$ for all $i$.  By (\ref{eqn.r}), $r_{n}=(n+1)r_{0}-nr_{-1}$ for all $n \in \mathbb{Z}$.  Thus, $r_{n}>0$ for all $n$ if and only if \begin{equation} \label{eqn.ineq0}
(n+1)r_{0}>nr_{-1}
\end{equation} 
for all $n$.  Rearranging this inequality for $n>0$, we find 
\begin{equation} \label{eqn.ineq1}
\frac{n+1}{n}>\frac{r_{-1}}{r_{0}}.
\end{equation}
Substituting $-n$ for $n$ in (\ref{eqn.ineq0}) yields $(-n+1)r_{0}>-nr_{-1}$.  If $n>1$, this may be rearranged to get 
\begin{equation} \label{eqn.ineq2}
\frac{r_{-1}}{r_{0}} > \frac{n-1}{n}.
\end{equation}
and so, considering both inequalities (\ref{eqn.ineq1}) and (\ref{eqn.ineq2}), we deduce $r_{0}=r_{1}$.  

Next, we note that
\begin{eqnarray*}
2 & = & \operatorname{dim}\operatorname{Hom}(\mathcal{E}_{-1},\mathcal{E}_{0}) \\
& = & d_{0}r_{-1}-d_{-1}r_{0} \\
& = & r_{0}(d_{0}-d_{-1}),
\end{eqnarray*}
so that the only possible values of rank are $1$ or $2$.  Finally, if both $r_{-1}$ and $r_{0}$ are $2$, then $d_{0}=d_{-1}+1$ and $r_{i}=2$ for all $i$.  It follows that $d_{i}=d_{i-1}+1$ for all $i$, so if there was a helix in this case, there would be a term in the helix with rank and degree both two.  Since this bundle isn't simple, we do not obtain an elliptic helix.

The converse, that if $r_{-1}=1=r_{0}$ and $d_{0}=d_{-1}+2$, then $(\mathcal{E}_{-1}, \mathcal{E}_{0})$ extends to an elliptic helix follows from the fact that $r_{i}=1$ for all $i\in \mathbb{Z}$ in this case, so that $r_{i}>0$ for all $i$.

Finally, suppose $d>2$.  By (\ref{eqn.a}) we have
$$
r_{m}=\alpha_{+}^{m+1}r_{+}+\alpha_{-}^{m+1}r_{-}
$$
for all $m \in \mathbb{Z}$.  Thus, if $r_{m}>0$ for all $m \in \mathbb{Z}$, then, since  $\alpha_{+}>1$ and $\alpha_{+}\alpha_{-}=1$, $r_{+}\alpha_{+}^{2m+2}>-r_{-}$ for all $m \in \mathbb{Z}$.  Since the left-hand side of this inequality is unbounded as $m \to \infty$, this is only possible if either $r_{+} > 0$ or $r_{+}=0$ and $r_{-}>0$.  The latter case is impossible since if $r_{+}=0$, then by (\ref{eqn.evectorexpand}), $r_{-}$ is irrational and equal to $r_{-1}$, which is an integer.  Thus, $r_{+}>0$.  Similarly, $r_{+}\alpha_{+}^{2m+2}>-r_{-}$ for all $m << 0$ implies that $r_{-}>0$ since $\alpha_{+}\alpha_{-}=1$.  We conclude that if $r_{m}>0$ for all $m \in \mathbb{Z}$, then $r_{+}>0$ and $r_{-}>0$.  Using (\ref{eqn.a}) in case $m=-1$ and $m=0$, we deduce that
\begin{equation} \label{eqn.rplus}
r_{+}=\frac{r_{0}-\alpha_{-}r_{-1}}{\alpha_{+}-\alpha_{-}}
\end{equation}
so that $r_{+}>0$ implies $r_{0}>\alpha_{-}r_{-1}$, which implies $\alpha_{+}>\frac{r_{-1}}{r_{0}}$.  Similarly, since $r_{-}>0$, we find $(\alpha_{+}-\alpha_{-})r_{-1}>r_{0}-\alpha_{-}r_{-1}$, which implies $\frac{r_{-1}}{r_{0}}>\alpha_{-}$.

Conversely, suppose
\begin{equation} \label{eqn.conbigd}
\alpha_{+} > \frac{r_{-1}}{r_{0}} > \alpha_{-}.
\end{equation}
We first claim that $r_{+}>0$ and $r_{-}>0$.  To prove that $r_{+}>0$, we know that since $\alpha_{+}r_{0}>r_{-1}$ by (\ref{eqn.conbigd}), we have
$$
\alpha_{+}(\alpha_{+}r_{+}+\alpha_{-}r_{-})> r_{+}+r_{-}
$$
by (\ref{eqn.evectorexpand}).  Simplifying this inequality yields $(\alpha_{+}^{2}-1)r_{+}>0$, whence the first claim.

Next, we prove that $r_{-}>0$.  We first note that it suffices to prove $r_{-1}>r_{+}$ since, if this were true, then the fact that $r_{-1}=r_{+}+r_{-}$ implies that $r_{-}>0$.  Thus, we prove that $r_{-1}>r_{+}$.  To prove this, we note that it suffices to prove $\alpha_{-}r_{0}>r_{+}$.  For, if this is the case, then by (\ref{eqn.conbigd}), $r_{-1}>\alpha_{-}r_{0}>r_{+}$ as desired.  Thus, to prove $r_{-}>0$, it remains to prove $\alpha_{-}r_{0}>r_{+}$.  To this end, by (\ref{eqn.evectorexpand}), we have 
$$
r_{+}=\frac{r_{0}-\alpha_{-}r_{-1}}{\alpha_{+}-\alpha_{-}}
$$
so we need to prove 
$$
(\alpha_{+}-\alpha_{-})\alpha_{-}r_{0}>r_{0}-\alpha_{-}r_{-1}.
$$
Simplifying this last inequality yields $\frac{r_{-1}}{r_{0}}>\alpha_{-}$, which holds by (\ref{eqn.conbigd}).

Thus, we have established that $r_{+}$ and $r_{-}$ are positive, and the result now follows from (\ref{eqn.a}).
\end{proof}

\begin{remark} \label{rem.invariant}
Since $\alpha_{\pm}$ are the roots of the quadratic $q(\lambda) = \lambda^2 - d \lambda +1$, the condition (\ref{eqn.helixcond}) from Theorem \ref{thm.elliptic2extend} amounts to $0 > q(\tfrac{r_{-1}}{r_0})$ which in turn can be re-written more succinctly as $dr_{-1}r_0 -(r_{-1}^2+r_{0}^{2})> 0$.  That is, the integer 
$$
D(r_{-1},r_{0}):= dr_{-1}r_{0}-r_{-1}^{2}-r_{0}^{2}
$$ 
is positive.
\end{remark}

We now show that the elliptic helices described in Theorem \ref{thm.elliptic2extend} with $d>2$ have an irrational limit slope.

\begin{lemma} \label{lemma.limit}
Let $\underline{\mathcal{E}}$ denote a two-periodic elliptic helix with $d>2$. Let $\mu_{n}$ denote the slope of $\mathcal{E}_{n}$.  Then $\lim_{n \to -\infty}\mu_{n}$ exists and is irrational.
\end{lemma}

\begin{proof}
By (\ref{eqn.a}), we note that since $\alpha_{+}>1>\alpha_{-}$,
$$
\lim_{n \to -\infty}\mu_{n} = \frac{d_{-}}{r_{-}}.
$$
Furthermore, we may solve for $r_{-}$ and $d_{-}$ given the ranks and degrees of $\mathcal{E}_{-1}$ and $\mathcal{E}_{0}$.  We find that
$$
d_{-}=d_{-1}(\frac{1}{2}+\frac{d}{2\sqrt{d^2-4}})-d_{0}\frac{1}{\sqrt{d^2-4}}
$$
and
$$
r_{-}=r_{-1}(\frac{1}{2}+\frac{d}{2\sqrt{d^2-4}})-r_{0}\frac{1}{\sqrt{d^2-4}}
$$
so that the negative limit slope exists and is
$$
\frac{d_{-}}{r_{-}}=\frac{d_{-1}(\frac{1}{2}+\frac{d}{2\sqrt{d^2-4}})-d_{0}\frac{1}{\sqrt{d^2-4}}}{r_{-1}(\frac{1}{2}+\frac{d}{2\sqrt{d^2-4}})-r_{0}\frac{1}{\sqrt{d^2-4}}}.
$$
A straightforward computation, using the fact that $d=d_{0}r_{-1}-d_{-1}r_{0}$, shows that
\begin{multline} \label{eqn.neglimit}
\theta(d_{-1},d_0,r_{-1},r_0) :=
\frac{d_{-}}{r_{-}}  \\
=\frac{2(r_{0}d_{0}+d_{-1}r_{-1})-d(r_{0}d_{-1}+d_{0}r_{-1})}{2(r_{0}^2+r_{-1}^2-dr_{0}r_{-1})}+\frac{d}{2(r_{0}^2+r_{-1}^2-dr_{0}r_{-1})}\sqrt{d^2-4}.
\end{multline}
By Remark \ref{rem.invariant}, the denominator is nonzero, and by hypothesis, $d>2$.  Thus, since $\sqrt{d^2-4}$ is irrational, the result follows.
\end{proof}

\begin{remark}  \label{rem.Dfunction}
In the formula for the limit slope $\theta$ in Lemma~\ref{lemma.limit}, the denominator contains the quadratic function of ranks $D(r_{-1},r_0)$ defined in Remark \ref{rem.invariant}.  Note that if the values for $\theta$ and $d$ are fixed, then so is the value of $D(r_{-1},r_0)$ since $\sqrt{d^2-4}$ is irrational.
\end{remark}

\section{A condition for ampleness}
In this section, we utilize aspects of Polishchuk's argument from \cite[Theorem 3.5]{polish} to prove that two-periodic elliptic helices are ample.  The argument we give is general enough to allow us to deduce that the helices we constructed in \cite[Theorem 7.23]{morph} are ample, as well.  Although we stated this fact in \cite{morph}, we did not include details there.

For the readers convenience, we recall the definition of ampleness.  The definition is (superficially) distinct from that in \cite{polish}, as our indexing conventions are different.  We let $\underline{\mathcal{L}} = (\mathcal{L}_{i})_{i \in \mathbb{Z}}$ denote a sequence of objects in ${\sf C}$ such that $\operatorname{Hom}(\mathcal{L}_{i},\mathcal{L}_{i})=k$.

We call $\underline{\mathcal{L}}$
\begin{itemize}
\item{} {\it projective} if for every epimorphism $f: \mathcal{M} \rightarrow \mathcal{N}$ in ${\sf C}$ there exists an integer $n$ such that $\operatorname{Hom}_{\sf C}(\mathcal{L}_{-i},f)$ is surjective for all $i>n$, and

\item{} {\it ample for ${\sf C}$} if it is projective, and if for every $\mathcal{M} \in {\sf C}$ and every $m \in \mathbb{Z}$ there exists a surjection
    $$
    \oplus_{j=1}^{s}\mathcal{L}_{-i_{j}} \rightarrow \mathcal{M}
    $$
for some $i_{1},\ldots, i_{s}$ with $i_{j} \geq m$ for all $j$.
\end{itemize}

\begin{lemma} \label{lemma.ample}
Suppose $\theta$ is irrational, and suppose $\underline{\mathcal{L}} := (\mathcal{L}_{n})_{n \in \mathbb{Z}}$ is a sequence of simple bundles over $X$.  Let $r_{n}$ and $\mu_{n}$ denote the rank and the slope of $\mathcal{L}_{n}$.  If
\begin{enumerate}
\item{} $\mu_{n} > \theta$ for all $n$,

\item{} $\lim_{n \to -\infty}\mu_{n} = \theta$ and

\item{} $\lim_{n \to -\infty} r_{n}^{2}(\mu_{n}-\theta)$ exists and is greater than $1$,
\end{enumerate}
then $\underline{\mathcal{L}}$ is ample for ${\sf C}^{\theta}$.
\end{lemma}

\begin{proof}
In what follows, we will write $d_{n}$ for the degree of $\mathcal{L}_{n}$.  By the proof of \cite[Theorem 3.5]{polish}, two inequalities must be satisfied.  First, we must have
$$
\mu_{n}-\theta > \frac{1}{r_{n}^{2}}
$$
for $n<<0$ (this is the $r=0$ case in Polishchuk's proof), and, for all integer pairs $(r,d)$ with $r>0$ and $\mu:=d/r>\theta$,
$$
r_{n}^{2}(\mu_{n}-\theta)> \frac{\mu-\theta}{\mu-\mu_{n}}
$$
for $n<<0$.  Since $\lim_{n \to -\infty} \frac{\mu-\theta}{\mu-\mu_{n}}=1$, both of these inequalities hold by our hypotheses.
\end{proof}

The hypotheses in the next lemma are purely numerical, but the notation is consistent with the $d>2$ case in Section \ref{sub.notation}.
\begin{lemma}  \label{lem:amplenessNumerics}
Let $d \in \mathbb{R}$ be such that $|d| >2$ and $\alpha_{\pm}$ be the roots of $\lambda^2 - d\lambda + 1$ indexed so $\alpha_+>1>\alpha_-$. Let $d_{\pm}, r_{\pm} \in \mathbb{R}$ and set $d_{n} = d_+ \alpha_+^{n+1} + d_-\alpha_-^{n+1}$ and $r_{n} = r_+ \alpha_+^{n+1} + r_-\alpha_-^{n+1}$ for $n \in \mathbb{Z}$. Finally, suppose $r_{n} \neq 0$ for all $n$, $\theta := \lim_{n \to -\infty} d_n/r_n$ and
$$ M = \begin{pmatrix}
    r_{-1} & d_{-1} \\ r_{0} & d_{0}
\end{pmatrix}.
$$
If $\det M > \sqrt{d^2-4}$ then $\lim_{n \to -\infty} r^2_n (d_n/r_n - \theta)$ exists and is greater than 1.
\end{lemma}
\begin{proof}
We begin by noting that since $\alpha_{+}\alpha_{-}=1$, $\lim_{n \to -\infty}\frac{d_{n}}{r_{n}}=\frac{d_{-}}{r_{-}}$.  Next, we observe that
\begin{eqnarray*}
r_{n}^{2}(d_{n}/r_{n} - \theta) & = & r_{n}(d_{n}-r_{n}\theta) \\
& = & (r_{+}\alpha_{+}^{n+1}+r_{-}\alpha_{-}^{n+1})\biggl((d_{+}\alpha_{+}^{n+1}+d_{-}\alpha_{-}^{n+1})-(r_{+}\alpha_{+}^{n+1}+r_{-}\alpha_{-}^{n+1})\frac{d_{-}}{r_{-}}\biggr) \\
& = & (r_{+}\alpha_{+}^{n+1}+r_{-}\alpha_{-}^{n+1})\biggl(d_{+}\alpha_{+}^{n+1}-\frac{r_{+}d_{-}}{r_{-}}\alpha_{+}^{n+1}\biggr) \\
& = & (r_{+}\alpha_{+}^{n+1}+r_{-}\alpha_{-}^{n+1})\alpha_{+}^{n+1}\biggl(d_{+}-\frac{r_{+}d_{-}}{r_{-}}\biggr)
\end{eqnarray*}
Thus,
$$
\lim_{n \to \infty}r_{-n}^{2}(\mu_{-n}-\theta) = r_{-}d_{+}-r_{+}d_{-} = \det D
$$
where
$$
D = \begin{pmatrix}
    r_- & d_- \\ r_{+} & d_{+}
\end{pmatrix}.
$$
Now
$$\begin{pmatrix}
    1 & 1 \\ \alpha_- & \alpha_+
\end{pmatrix} D = M
$$
so the condition $\det D >1$ just amounts to the condition
$$ \det M > \begin{vmatrix} 1 & 1 \\ \alpha_- & \alpha_+
\end{vmatrix} = \alpha_+ - \alpha_- = \sqrt{d^2 -4}.
$$

\end{proof}

\begin{prop} \label{prop.smallample}
Let $\underline{\mathcal{E}}$ be a two-periodic elliptic helix with negative limit slope $\theta$.  Then $\underline{\mathcal{E}}$ is ample for ${\sf C}^{\theta}$.
\end{prop}

\begin{proof}
Since $\operatorname{det }M = \operatorname{det }\begin{pmatrix}
    r_{-1} & d_{-1} \\ r_{0} & d_{0}
\end{pmatrix} =d> \sqrt{d^{2}-4}$, this follows immediately from Lemma \ref{lem:amplenessNumerics}.
\end{proof}

\begin{proposition}
The three-periodic elliptic helices $(\mathcal{E}_{n})_{n \in \mathbb{Z}}$ constructed in \cite[Theorem 7.23]{morph} are ample.
\end{proposition}

\begin{proof}

Note first that the formulas for $r_{n}$ and $d_{n}$ given in \cite[Lemma 7.18]{morph} are for $n \geq 0$, whereas we are interested in behavior for $n<<0$.  Now, for all $n$, we have
$$
a_{n+1}=da_{n}-da_{n-1}+a_{n-2}
$$
for $a_{j}=r_{j}$ or $a_{j}=d_{j}$.  It follows that
$$
\begin{pmatrix} d & -d & 1 \\ 1 & 0 & 0 \\ 0 & 1 & 0 \end{pmatrix} \begin{pmatrix} a_{n-1} \\ a_{n} \\ a_{n+1} \end{pmatrix} = \begin{pmatrix} a_{n-2} \\ a_{n-1} \\ a_{n} \end{pmatrix}.
$$
Note that the characteristic polynomial this time is
$$\lambda^3 - d \lambda^2 + d\lambda -1 = (\lambda -1) (\lambda^2 -(d-1)\lambda + 1).$$
Let $\alpha_{\pm}$ be the roots of $\lambda^2 - (d-1) \lambda + 1$ indexed so $\alpha_+ > 1 > \alpha_-$ so that the sequences of ranks and degrees are linear combinations of the geometric series $1, \alpha_+^n, \alpha_-^n$.
Note first that $(d_{0},r_{0})=(0, 1)$, $(d_{1}, r_{1})=(d,1)$ and $(d_{2},r_{2})=(d^{2}-d, d-2)$.  We deduce that $(d_{-1}, r_{-1})=(-d, d-2)$ and $(d_{-2},r_{-2})=(-d^{2}+d, d^{2}-3d+1)$.

Solving the initial value problem, we find that the sequences of $d_{-n}$ and $r_{-n}$ are actually only linear combinations of $\alpha_-^n$ and $\alpha_+^n$. Indeed, we have parallel cross products
$$
\begin{pmatrix}
    1 \\ \alpha_- \\ \alpha_-^2
\end{pmatrix} \times
\begin{pmatrix}
    1 \\ \alpha_+ \\ \alpha_+^2
\end{pmatrix} =
(\alpha_+-\alpha_-)\begin{pmatrix}
    1 \\ 1-d \\ 1
\end{pmatrix}, \quad
\begin{pmatrix}
    d_{0} \\ d_{-1} \\ d_{-2}
\end{pmatrix} \times
\begin{pmatrix}
    r_{0} \\ r_{-1} \\ r_{-2}
\end{pmatrix} =
d\begin{pmatrix}
    1 \\ 1-d \\ 1
\end{pmatrix}
$$

Hence, we may use Lemma~\ref{lem:amplenessNumerics} to verify the hypotheses of Lemma~\ref{lemma.ample}. We are done on noting
$$ \det \begin{pmatrix}
    r_{-1} & d_{-1} \\ r_{0} & d_{0}
\end{pmatrix} = \begin{vmatrix}
    d-2 & -d \\ 1 & 0
\end{vmatrix} = d > \sqrt{(d-1)^2 - 4}.
$$
\end{proof}

\section{Helix operations and their numerical incarnations} \label{section.helixops}
In this section, we consider operations on a helix $\underline{\mathcal{E}}$, and how it affects the numerical data $d_i = \deg \mathcal{E}_i, \ r_i = \rk \mathcal{E}_i, i = -1,0$. Of particular interest is the question, how does the negative limit slope function $\theta(d_{-1},d_0,r_{-1},r_0)$ in (\ref{eqn.neglimit}) change with the corresponding changes to $d_{-1},d_0,r_{-1},r_0$?

We can shift the helix $\underline{\mathcal{E}}$ to give $\underline{\mathcal{F}}$ with $\mathcal{F}_i = \mathcal{E}_{i+1}$. Of course, the corresponding values of $d = \dim \Hom(\mathcal{F}_{-1}, \mathcal{F}_0)$ and negative limit slope $\theta$ for $\underline{\mathcal{F}}$ are the same as those for $\underline{\mathcal{E}}$. A direct computation proves more generally the numerical version below.
\begin{proposition}  \label{prop.helixShift}
Let $d_{-1},d_0,r_{-1},r_0 \in \mathbb{N}$. Define
$$
\begin{pmatrix}
    d'_{-1} \\ d'_0
\end{pmatrix}
 = A
\begin{pmatrix}
    d_{-1} \\ d_0
\end{pmatrix}, \quad
\begin{pmatrix}
    r'_{-1} \\ r'_0
\end{pmatrix}
 = A
\begin{pmatrix}
    r_{-1} \\ r_0
\end{pmatrix},
$$
where $A=\begin{pmatrix} 0 & 1 \\ -1 & d \end{pmatrix}$ as in Section \ref{sub.notation}.  Then $\theta(d'_{-1},d'_0,r'_{-1},r'_0) = \theta(d_{-1},d_0,r_{-1},r_0)$ and
$$
\begin{vmatrix}
    r'_{-1} & d'_{-1} \\ r'_0 & d'_0
\end{vmatrix} =
\begin{vmatrix}
    r_{-1} & d_{-1} \\ r_0 & d_0
\end{vmatrix}.
$$
\end{proposition}

We can also tensor a helix by a line bundle $\cL$, say of degree $a$. Clearly we have $\Hom(\cL \otimes \cE_{-1},\cL \otimes \cE_0) \simeq \Hom(\cE_{-1},\cE_0)$ and the new negative limit slope has increased by the integer $a$. The numerical version of this is the following.

\begin{proposition}  \label{prop.twistHelix}
Let $\cvecr\in \mathbb{Z}_{>0}^2$ and $\cvecd,\cvecdp \in \mathbb{Z}^2$. Then
$$
\begin{vmatrix}
    r_{-1} & d'_{-1} \\ r_0 & d'_0
\end{vmatrix} =
\begin{vmatrix}
    r_{-1} & d_{-1} \\ r_0 & d_0
\end{vmatrix}
$$
iff there is $a \in \mathbb{Q}$ such that $\cvecdp = \cvecd + a \cvecr$ in which case
$$
\theta(d'_{-1},d'_0,r_{-1},r_0) = a + \theta(\seed).
$$
\end{proposition}

The final operation we consider on a helix is taking its dual $\underline{\cE}^*$ (and reversing the order of the sequence so its $n$-th term is $\cE^*_{-n}$). Clearly, the dimension $d$ of the Hom-spaces remains the same, as does the quadratic function on ranks $D(r_{-1},r_0)$ defined in Remark~\ref{rem.invariant}. However, the negative limit slope changes in a rather interesting way as the following proposition shows.

\begin{proposition}   \label{prop.dualHelix}
Let $d_{-1},d_0 \in \mathbb{Z}^2$ and $r_{-1},r_0$ be positive integers and
$d = \begin{vmatrix}
    r_{-1} & d_{-1} \\ r_0 & d_0
\end{vmatrix}$ as usual. If
$\theta(\seed) = a + b\sqrt{d^2 - 4}$ with $a, b \in \mathbb{Q}$, then $\theta(-d_0,-d_{-1},r_0,r_{-1}) = -a + b \sqrt{d^2 - 4}$
\end{proposition}

\section{Helices with given $d$ and negative limit slope $\theta$} \label{section.diffcovers}
Suppose we fix a positive integer $d$ and slope $\theta\in \mathbb{Q} + \mathbb{Q}\sqrt{d^2-4}$. Consider the question of classifying ``double covers'' $f \colon {\sf C}^{\theta} \to \mathbb{P}^1_d$.  A natural place to start is to classify all helices $\underline{\cE}$ with negative limit slope $\theta$ and consecutive Hom spaces $d$-dimensional up to numerical class. In this section, we give a version of this numerical classification for some easy cases.

We fix a positive integer $d$ and start with a ``numerical seed'' consisting of two vectors $\cvecr \in \mathbb{Z}^2_{>0},  \cvecd \in \mathbb{Z}^2$ such that $\begin{vmatrix}
    r_{-1} & d_{-1} \\ r_0 & d_0
\end{vmatrix} = d>2$. Note that helices with $d=2$ were numerically classified already in Theorem~\ref{thm.elliptic2extend} so we exclude this case. We say $\cvecr,\cvecd$ are {\em relatively prime} if $1 = \gcd(r_{-1},d_{-1}) = \gcd(r_0,d_0)$. For the moment, we do not fix $\theta$, but we do fix \begin{equation}\label{eq.solveDequals}
D := D(r_{-1},r_0) = dr_{-1}r_0 - r_{-1}^2 - r^2_0 >0
\end{equation}
which would be fixed if $\theta$ were.

We extend $\cvecr,\cvecd$ to sequences as in Section \ref{sub.notation} by
$$
\cvec{r_{n-1}}{r_n} = A^n \cvecr, \quad \cvec{d_{n-1}}{d_n} = A^n \cvecd \quad \text{for }\ n \in \mathbb{Z},
$$
where $A=\begin{pmatrix} 0 & 1 \\ -1 & d \end{pmatrix}$.  Recall from Theorem~\ref{thm.elliptic2extend} that the condition $D>0$ is equivalent to the fact that $(r_{i})_{i \in \mathbb{Z}}$ is a sequence of positive integers and hence, that $(r_{i})_{i \in \mathbb{Z}}, (d_{i})_{i \in \mathbb{Z}}$ are the sequences of ranks and degrees of a 2-periodic elliptic helix if $\cvecr, \cvecd$ are relatively prime.

We first consider the Diophantine equation (\ref{eq.solveDequals}) for fixed $D$. One might try to solve this using techniques in the theory of Pell-like equations as described in the introduction, but here, the geometry gives us some extra symmetry. Indeed, we know from Proposition~\ref{prop.helixShift}, that shifting a helix does not change the quantity $D$, so any solution $\cvecr$ yields the infinite family of solutions $A^n\cvecr, n \in \mathbb{Z}$. Amongst all these ``shifted'' solutions, we can more or less select a unique one as follows. Since we have assumed $d>2$ we know from Equation~(\ref{eqn.a}) that
$$r_n = r_+ \alpha_+^{n+1} + r_- \alpha_-^{n+1}$$
where $\alpha_{\pm}$ are the roots of the characteristic polynomial $cp_A(\lambda) = \lambda^2 - d\lambda +1$ and $r_{\pm}$ are positive reals. This is a concave up function of $n$ so, at the price of shifting the sequence, we can normalise until $r_0$ is a minimum. Note that if there is another minimum, then it has to be either $r_{-1}$ or $r_1$. We thus consider the {\em minimality criterion} $r_0 \leq r_{-1}, r_1 = dr_0 - r_{-1}$ which can be more simply expressed as
\begin{equation}   \label{eq.minimalr0}
r_0 \leq r_{-1} \leq (d-1) r_0.
\end{equation}
Note that this minimality criterion on $\cvecr$ ensures that $D(r_{-1},r_0) >0$. It can also be written as
\begin{equation}  \label{eq.minimalV2}
1 \leq \frac{r_{-1}}{r_0}  \leq d-1
\end{equation}
This can be used to limit solutions to the Diophantine equation~(\ref{eq.solveDequals}).
\begin{lemma}  \label{lem.smallDcase}
Suppose that $0<D < 4(d-2)$. If $\cvecr\in \mathbb{Z}_{>0}^2$ satisfies
$$dr_0r_{-1} - r_{-1}^2-r_0^2 = D.$$
Then there are integers $n, y$ such that $\cvecr = A^n \cvec{y}{1}$ where $dy - y^2 - 1 = D$. Furthermore, $D \geq d-2$.
\end{lemma}
\begin{proof}
Since shifts correspond to multiplication by $A$, we may assume $r_0$ satisfies the minimality criterion of (\ref{eq.minimalV2}). We wish to establish the first assertion by proving that $r_0 = 1$. Note
\begin{equation}  \label{eq.dividedD}
\frac{D}{r_0^2} = -\left(\frac{r_{-1}}{r_0}\right)^2 + d \left(\frac{r_{-1}}{r_0}\right) - 1 =-cp_A\left(\frac{r_{-1}}{r_0}\right)
\end{equation}
and that on the domain $[1,d-1]$ given by (\ref{eq.minimalV2}), $-cp_A(\lambda)$ has its minima at the endpoints.  In addition, this minimum value is $d-2$. Thus
\begin{equation} \label{eq.Dinequality}
    \frac{D}{r_0^2} \geq d-2
\end{equation}
This firstly gives $D \geq d-2$. Furthermore, if $r_0 \geq 2$ then $D \geq 4(d-2)$ so we must have $r_0 = 1$.
\end{proof}

This suggests that we have a lot more control when $D$ is small relative to $d$. We will mainly be interested in the case where $D < 4(d-2)$ so $r_0 = 1$ is fixed.

Let us suppose now that $\cvecr \in \mathbb{Z}_{>0}^2$ is given and pose the question, when can you find a relatively prime $\cvecd \in \mathbb{Z}^2$ such that
$
\begin{vmatrix}
    r_{-1} & d_{-1} \\ r_0 & d_0
\end{vmatrix} = d
$. The following gives a necessary criterion.
\begin{proposition}  \label{prop.necCriterionExistd}
Let $\cvecr \in \mathbb{Z}_{>0}^2, \cvecd \in \mathbb{Z}^2$ be relatively prime and $d = \begin{vmatrix}
    r_{-1} & d_{-1} \\ r_0 & d_0
\end{vmatrix}$.
Then $\gcd(r_{-1},r_0) = \gcd(r_{-1},d) = \gcd(r_0,d)$.
\end{proposition}
\begin{proof}
Note first that $\gcd(r_{-1},r_0)$ divides both $\gcd(r_{-1},d),\gcd(r_0,d)$. Conversely, suppose $a:= \gcd(r_0,d)$ also divides $r_{-1}$. Then $a|r_{-1}d_0 - r_0d_{-1}$ and hence $a | r_{-1}d_0$. Now $a,d_0$ are relatively prime so $a|r_{-1}$ too and the proof is now complete noting symmetry in $r_{-1},r_0$.
\end{proof}

We unfortunately only have a partial converse to the above result.
\begin{prop}  \label{prop.suffCriterionExistd}
Suppose $d\in \mathbb{Z}$ and $\cvecr \in \mathbb{Z}_{>0}^2$ are such that $\gcd(r_{-1},r_0) = \gcd(r_{-1},d) = \gcd(r_0,d)=: \bar{r}$ is odd. Then there exists $\cvecd \in \mathbb{Z}^2$ relatively prime to $\cvecr$ such that
\begin{equation}  \label{eq.suffCritd}
d = \begin{vmatrix}
    r_{-1} & d_{-1} \\ r_0 & d_0
\end{vmatrix}
\end{equation}
If $\bar{r} = 1$ then any solution $\cvecd$ to Equation~(\ref{eq.suffCritd}) is automatically relatively prime to $\cvecr$.
\end{prop}
\begin{proof}
The last assertion follows immediately from definitions.

Write $\cvecr = \bar{r} \cvec{r'_{-1}}{r'_0}, d = \bar{r} d'$. Since $r'_{-1}, r'_0,d'$ are pairwise relatively prime, we can find $\cvec{d'_{-1}}{d'_0} \in \mathbb{Z}^2$ relatively prime to $\cvec{r'_{-1}}{r'_0}$ such that  $d' = \begin{vmatrix}
    r'_{-1} & d'_{-1} \\ r'_0 & d'_0
\end{vmatrix}$. We seek an integer $m$ such that $\cvecd = m \cvec{r'_{-1}}{r'_0} + \cvec{d'_{-1}}{d'_0}$ has coordinates relatively prime to $\bar{r}$. Since any such $\cvecd$ satisfies Equation~(\ref{eq.suffCritd}) and is relatively prime to $\cvec{r'_{-1}}{r'_0}$, we will have obtained our desired $\cvecd$.

For each prime $p|\bar{r}$ and $i = -1,0$, the fact that $r'_i,d'_i$ are relatively prime ensures there is at most one value of $m(p)_i \in \mathbb{Z}/p$ such that
$$
m(p)_i r'_i + d'_i \equiv 0 \bmod p.
$$
Since $p>2$, we can pick $m(p) \in \mathbb{Z}/p$ such that $m(p) \not\equiv m(p)_i \bmod p$ for all $p$. It suffices now to use the Chinese remainder theorem to pick any $m \in \mathbb{Z}$ such that $m \equiv m(p) \bmod p$ for the finite number of primes $p| \bar{r}$. 
\end{proof}

\begin{remark}
It is easy to see that the above Proposition fails if the assumption that $\bar{r}$ is odd does not hold.
\end{remark}

We are interested in helices up to shifting and twisting by line bundles (see Definition \ref{def:numericalShift}). Recall from Propositions~\ref{prop.helixShift}, \ref{prop.twistHelix} that such operations do not change the dimensions $d$ of consecutive Hom spaces, nor the invariant $D$, and that the negative limit slope $\theta$ will only change by the integer corresponding to the degree of the line bundle one twists by.

We will show later, in contrast to the commutative case $d=2$, that $\theta$ and $d$ do not uniquely determine the numerical class of a helix. The only uniqueness result we have is the following.

\begin{proposition}  \label{prop.helixUniqueUpToDual}
Let $\underline{\cE}, \underline{\cE}'$ be helices with ranks and degrees given by $r_i, d_i, r'_i, d'_i$ respectively and let $\theta$ be the negative limit slope of $\underline{\cE}$. Suppose that i) both helices have $d$-dimensional consecutive Hom spaces, ii) $D:= D(r_{-1},r_{0}) = D(r'_{-1},r'_0)$, iii) $\gcd(r_{-1},r_0) = 1$ and iv) $r_i = r'_j$ for some $i,j$.

Then
\begin{enumerate}
    \item $\underline{\cE}'$ is in the same numerical class as either $\underline{\cE}$ or its dual $\underline{\cE}^*$.
    \item $\theta$ is also the negative limit slope of some twist of $\underline{\cE}^*$ iff $\theta \in \frac{1}{2}\mathbb{Z} - \frac{d}{2D}\sqrt{d^2-4}$.
    \item If $D<4(d-2)$, $\underline{\cE}^*, \underline{\cE}$ are in the same numercial class iff $r_{i-1} = r_i =1$ or $r_{i-1} = r_{i+1}$ for some $i$.
\end{enumerate}
\end{proposition}
\begin{proof}
The solutions of the equation
\begin{equation*}
    - y^2 + dr_0 y - r_0^2 = D
\end{equation*}
are $y=r_{-1}$ and $y=r_{1}$.  After shifting, we may assume, by iv), that $r_{0}'=r_{0}$ so that either $r_{-1}'=r_{-1}$ or $r_{-1}'=r_{1}$.  Proposition~\ref{prop.twistHelix} and our assumption iii) ensure that $\underline{\cE}'$ is numerically equal to a twist of $\underline{\cE}$ in the first case, and a twist of $\underline{\cE}^*$ in the second. This proves (1).

If $\theta = a - \frac{d}{2D}\sqrt{d^2-4}$, then the negative limit slope of twists of $\underline{\cE}^*$ must have the form $n -a - \frac{d}{2D}\sqrt{d^2-4}$ for some integer $n$ by Proposition \ref{prop.dualHelix} and Proposition \ref{prop.twistHelix}. These can only be equal if $a \in \frac{1}{2}\mathbb{Z}$, thus proving (2).

Finally, suppose that $\underline{\cE}^*,\underline{\cE}$ are in the same numerical class. We first shift $\underline{\cE}$ so that $r_0\leq r_{-1}$ and $r_0 < r_1$. If $r_{-1} \neq r_0$, then we must have $r_{-1} = r_1$. The converse, is just an elementary calculation.
\end{proof}

We now address the question of determining numerical classes of helices with given parameters $D$ and $d$. We wish to use Lemma~\ref{lem.smallDcase}, valid when $D < 4(d-2)$, which allows us to shift the helix so that $r_0 = 1$. We may thus suppose that the helix has $\cvecr = \cvec{y}{1}$ for some $y \in \mathbb{Z}_{>0}$. Note $D = dy-y^2-1$ which suggests that we should look at the cases $y=1,2,3,4$. We look at each of these in turn. Recall also from Proposition~\ref{prop.suffCriterionExistd}, given any $d$ relatively prime to $y$, one can find $\cvecd \in \mathbb{Z}^2$ relatively prime to $\cvecr$ with $\seedmatrix = d$. There thus exist helices corresponding to these numerical invariants.

The $y=1$ case is the following.
\begin{proposition}  \label{prop.caseyis1}
Fix an integer $d >2$. Let  $\cvecr \in \mathbb{Z}^2_{>0}, \cvecd \in \mathbb{Z}^2$ be relatively prime vectors such that $D(r_{-1},r_0) = d-2$ and $d = \seedmatrix$. Then $\cvecr = A^n\cvec{1}{1}$ for some $n\in \mathbb{Z}$.

In particular, all helices with $d$-dimensional consecutive Hom spaces and negative limit slope $\theta \in \mathbb{Q} -\frac{d}{2(d-2)}\sqrt{d^2-4}$ are in the same numerical class.
\end{proposition}
\begin{proof}
By Lemma~\ref{lem.smallDcase}, we may assume that $r_0 = 1$.  Since $D(r_{-1},r_{0})=d-2$, either $r_{-1}=1$ or $r_{-1}=d-1$.  If $r_{-1}=1$, we are done with the first part of the result, so suppose $r_{-1}=d-1$.  Then $r_{1}=dr_{0}-r_{-1}=1$, so, after shifting again, we may assume that $r_{-1}=r_{0}=1$, and the first assertion follows.

To prove the second assertion, suppose $\underline{\mathcal{E}}$, $\underline{\mathcal{E}}'$ both have negative limit slope in the indicated range and $d$-dimensional consecutive Hom spaces.  By our negative limit slope assumption and Remark \ref{rem.Dfunction}, $D(r_{-1},r_{0})=D(r_{-1}',r_{0}')=d-2$.  Thus, by Proposition~\ref{prop.helixUniqueUpToDual}(1), $\underline{\mathcal{E}}'$ is in the same numerical class as either $\underline{\mathcal{E}}$ or $\underline{\mathcal{E}}^{*}$.  However, Proposition~\ref{prop.helixUniqueUpToDual}(3), shows that $\underline{\cE}, \underline{\cE}^*$ are in the same numerical class so we are done. 
\end{proof}

The $y=2$ case is the following.
\begin{proposition}  \label{prop.caseyis2}
Fix an integer $d >2$. Let  $\cvecr \in \mathbb{Z}^2_{>0}, \cvecd \in \mathbb{Z}^2$ be relatively prime vectors such that $D(r_{-1},r_0) = 2d-5$ and $d = \seedmatrix$. Then $\cvecr = A^n\cvec{2}{1}$ or $A^n \cvec{d-2}{1}$ for some $n\in \mathbb{Z}$.

Let $\theta \in \mathbb{Q} - \frac{d}{2(2d-5)}\sqrt{d^2-4}$. If $d \neq 5$, then all helices $\underline{\cE}$ with $d$-dimensional consecutive Hom spaces and negative limit slope $\theta$ are in the same numerical class. For $d=5$, there are no helices unless the negative limit slope $\theta \in \mathbb{Z}+  \frac{1}{2} - \frac{1}{2}\sqrt{21}$ in which case there are two distinct numerical classes.
\end{proposition}
\begin{proof}
The proof of the first assertion follows the same argument as that for Proposition~
\ref{prop.caseyis1}.  To prove the second assertion, let $\theta'  \in \mathbb{Q} - \frac{d}{2(2d-5)}\sqrt{d^2-4}$, and suppose $\underline{\mathcal{E}}'$, denotes a two-periodic elliptic helix with negative limit slope $\theta'$.  Then $D(r_{-1}',r_{0}')=2d-5$.  On the other hand, consider the two-periodic elliptic helix $\underline{\mathcal{E}}$ with $\cvecr = \cvec{2}{1}, \cvecd = \cvec{-d}{0}$.  Note that $D(2,1) = 2d-5$ and that $\begin{vmatrix}
    2 & -d \\ 1 & 0
\end{vmatrix} = d$.  Then, by Proposition~\ref{prop.helixUniqueUpToDual}(1), $\underline{\mathcal{E}}'$ is in the same numerical class as either $\underline{\mathcal{E}}$ or $\underline{\mathcal{E}}^{*}$. Therefore, it suffices to show that if $d \neq 5$, then the negative limit slope of $\underline{\mathcal{E}}$, $\theta$, isn't the negative limit slope of some twist of $\underline{\mathcal{E}}^{*}$, while if $d=5$, $\underline{\mathcal{E}}$ and $\underline{\mathcal{E}}^{*}$ have the same negative limit slope, up to twisting, so that the assertion follows from the first part and Proposition~\ref{prop.helixUniqueUpToDual}(3).

To prove the claim above, we use Proposition~\ref{prop.helixUniqueUpToDual}(2) to check when $\theta \in \frac{1}{2}\mathbb{Z}- \frac{d}{2D}\sqrt{d^2-4}$.  From Equation~(\ref{eqn.neglimit}), the rational part of $\theta(\seed)$ is $-\frac{4d - d^2}{2(2d-5)}$ so we must solve for the condition $
2d-5 | 4d - d^2$. Now by Propostion~\ref{prop.necCriterionExistd}, we must have $d$ odd so we write $d = 2\bar{d} + 1$ and consider the divisibility condition
\begin{equation*}
4\bar{d} -3 \ \Bigm| \ 4 \bar{d}^2 - 4 \bar{d} -3.
\end{equation*}
Using the division algorithm, this is equivalent to $4 \bar{d} -3 | \bar{d} + 3$ from which it is easy to conclude that the only solutions are $\bar{d} = 1,2$ giving $d = 3,5$. The case $d=3$ is actually already covered in Proposition~\ref{prop.caseyis1} since then $\cvec{d-2}{1} = \cvec{1}{1}$. 
\end{proof}

\begin{example}
For an explicit example, the case $d=5, \theta = \frac{1}{2} - \frac{1}{2}\sqrt{21}$ can be obtained from helices with numerical seeds $\cvecr = \cvec{2}{1}, \cvecd = \cvec{-3}{1}$ or $\cvecr = \cvec{3}{1}, \cvecd = \cvec{-5}{0}$. They are not in the same numerical class.
\end{example}

The $y=3$ case is the following.
\begin{proposition}  \label{prop.caseyis3}
Fix an integer $d \geq 4$. Let  $\cvecr \in \mathbb{Z}^2_{>0}, \cvecd \in \mathbb{Z}^2$ be relatively prime vectors such that $D(r_{-1},r_0) = 3d-10$ and $d = \seedmatrix$. Then $\cvecr = A^n\cvec{3}{1}$ or $A^n \cvec{d-3}{1}$ for some $n\in \mathbb{Z}$.

Let $\theta \in \mathbb{Q} - \frac{d}{2(3d-10)}\sqrt{d^2-4}$. If $d \neq 5,10$, then all two-periodic elliptic helices $\underline{\cE}$ with $d$-dimensional consecutive Hom spaces and negative limit slope $\theta$ are in the same numerical class. For $d=10$, there are no two-periodic elliptic helices unless the negative limit slope $\theta \in \mathbb{Z} - \sqrt{6}$ in which case there are two numerical classes. The case $d=5$ is covered in Proposition~\ref{prop.caseyis2}.
\end{proposition}
\begin{proof}
The proof of the first assertion is similar to the proof of the first assertion in Proposition \ref{prop.caseyis1}.  For the second assertion, we repeat the proof used in Proposition~\ref{prop.caseyis2} using this time the vectors $\cvecr = \cvec{3}{1}, \cvecd = \cvec{-d}{0}$. Also, $d$ is coprime to $3$ so we write $d = 3\bar{d} + b$ where $b = 1$ or 2. The condition that the rational part of $\theta(\seed)$ is half an integer corresponds now to the divisibility condition
$$ 9 \bar{d} + 3b -10 \  \Bigm| \ 9 \bar{d}^2 + (6b - 18) \bar{d} + b^2 - 6b.$$
Examining cases $b=1,2$ one readily arrives at the unique solutions $d=4,5,10$. The solution $d=4$ is covered by Proposition~\ref{prop.caseyis1}, the solution $d=5$ is covered by Proposition~\ref{prop.caseyis2}  leaving the only new case $d=10$.
\end{proof}

\begin{example}
For an explicit example, the case $d=10, \theta = -\sqrt{6}$ can be obtained from helices with numerical seeds $\cvecr = \cvec{3}{1}, \cvecd = \cvec{-7}{1}$ or $\cvecr = \cvec{7}{1}, \cvecd = \cvec{-17}{-1}$. They are not in the same numerical class. 
\end{example}

The $y=4$ case is the following.
\begin{proposition}  \label{prop.caseyis4}
Fix an integer $d \geq 5$. Let  $\cvecr \in \mathbb{Z}^2_{>0}, \cvecd \in \mathbb{Z}^2$ be relatively prime vectors such that $D(r_{-1},r_0) = 4d-17$ and $d = \seedmatrix$. Then $\cvecr = A^n\cvec{4}{1}$ of $A^n\cvec{d-4}{1}$ for some $n\in \mathbb{Z}$.

In particular, all helices with $d$-dimensional consecutive Hom spaces and negative limit slope $\theta \in \mathbb{Q} -\frac{d}{2(4d-17)}\sqrt{d^2-4}$ are in the same numerical class unless $d= 17$. If $d= 17$, then there are two numerical equivalence classes for $\theta \in \mathbb{Z}  -\frac{d}{2(4d-17)}\sqrt{d^2-4}$.
\end{proposition}
\begin{proof}
We again repeat the proof used in Proposition~\ref{prop.caseyis2} using this time the vectors $\cvecr = \cvec{4}{1}, \cvecd = \cvec{-d}{0}$. We determine when $\theta(\seed)$ has rational part half an integer. Since $d$ must be odd, we may write $d = 4 \bar{d} + b$ where $b=1$ or 3. The required divisibility condition now is 
$$ 16 \bar{d} + 4b -17 \Big| 16 \bar{d}^2 + (8b-32) \bar{d} + b^2 - 8b.$$
Using the division algorithm, this is equivalent to 
\begin{equation}\label{eq:halfIntegerCond}
 16 \bar{d} + 4b -17 \Big| (4b+1) \bar{d} + b^2 - 4b -17.    
\end{equation}
Suppose first that $b=1$. Then we must have 
$16 \bar{d} - 13 \Big| 5\bar{d} - 20$. 
In particular, either $5\bar{d} -20 = 0$ which corresponds to $d = 17$, or we have the inequality 
\begin{equation}
|16 \bar{d} - 13 | \leq |5\bar{d} - 20|.
\end{equation}
Solving this gives $\bar{d} \leq \frac{11}{7}$.
For such $\bar{d}$, $16 \bar{d} - 13 \Big| 5\bar{d} - 20$ only when $\bar{d}=1$ giving the solution $d=5$ which is already covered in  Proposition~\ref{prop.caseyis1}. Note that when $d=17$, the rational part of $\theta(\seed)$ is integral. 

Suppose now that $b = 3$ so that Equation~(\ref{eq:halfIntegerCond}) becomes $16 \bar{d} -5 \Big| 13\bar{d}  -20$. Again we have an inequality, this time
\begin{equation}
|16 \bar{d} - 5 | \leq |13\bar{d} - 20|.
\end{equation}
Solving gives $\bar{d} \leq \frac{25}{29}$
so there are no solutions with $d \geq 5$. 
\end{proof}

\begin{example} 
We now consider the case $d=5$ and $\theta=\frac{23}{10}-\frac{1}{30}\sqrt{21}$.  In this case, $D=75$, so $D \neq d-2, 2d-5, 3d-10, 4d-17$.  In particular, our previous analysis cannot be used.  We let $\begin{pmatrix} r_{-1} \\ r_{0} \end{pmatrix} = \begin{pmatrix} 4 \\ 7 \end{pmatrix}$, $\begin{pmatrix} d_{-1} \\ d_{0} \end{pmatrix} = \begin{pmatrix} 9 \\ 17 \end{pmatrix}$, or $\begin{pmatrix} r_{-1} \\ r_{0} \end{pmatrix} = \begin{pmatrix} 5 \\ 5 \end{pmatrix}$, $\begin{pmatrix} d_{-1} \\ d_{0} \end{pmatrix} = \begin{pmatrix} 11 \\ 12 \end{pmatrix}$.  Since, in the second case, $5|r_{i}$ for all $i$, corresponding two-periodic elliptic helices are not in the same numerical class.  
\end{example}



\section{Double covers of noncommutative projective lines by noncommuative elliptic curves}
\label{sec.doubleCovers}

Modulo the automorphism group $PGL_2$ of $\mathbb{P}^1$, the finite 2-to-1 morphisms from $X$ to $\mathbb{P}^1$ are given by the complete linear systems of a degree two line bundle (Proposition \ref{prop.commutative}(2)), and the automorphism group of $X$ acts transitively on these (Proposition \ref{prop.commutative}(3)). The goal of this subsection is to investigate the corresponding noncommutative story where we replace $X$ with ${\sf C}^{\theta}$ and $\mathbb{P}^1$ with Piontkovski's noncommutative line $\mathbb{P}^1_d$. 

\subsection{Basic notions from noncommutative algebraic geometry}  Recall that a {\it $\mathbb{Z}$-algebra over $k$} is a ring $A$ with decomposition $A = \oplus_{i,j \in \mathbb{Z}} A_{ij}$ such that 
\begin{itemize}
\item{} $A_{ij}$ are vector spaces over $k$,
\item{} multiplication is induced by associative multiplication maps 
$$
A_{ij} \otimes_{k} A_{jl} \to A_{jl}
$$ 
(multiplication $A_{ij} A_{kl}=0$ if $j \neq k$), and 
\item{} each $A_{ii}$ contains a unit element $e_i$ satisfying the usual unit axiom.
\end{itemize}
See \cite{sierra} for more information about $\mathbb{Z}$-algebras.  If $A$ is a $\mathbb{Z}$-algebra, we let ${\sf Gr }A$ denote the category of graded right $A$-modules.  If, in addition, $A$ is coherent, we may define an abelian category
$$
{\sf proj }A := {\sf coh }A/{\sf tors }A
$$
where ${\sf coh }A$ denotes the full subcategory of graded right $A$-modules consisting of coherent modules and ${\sf tors }A$ is the full subcategory of $A$ consisting of right-bounded modules \cite{polishproj}.  In addition, we let $\pi: {\sf coh }A \rightarrow {\sf proj }A$ denote the quotient functor.

If $\underline{\mathcal{E}}$ is a sequence of objects in ${\sf C}$, then we let $B_{\underline{\mathcal{E}}}$ denote the $\mathbb{Z}$-algebra with $i,j$-component $\operatorname{Hom}(\mathcal{E}_{-j}, \mathcal{E}_{-i})$ and product induced by composition, and we let $\mathbb{S}^{nc}(\underline{\mathcal{E}})$ is the quadratic part of $B_{\underline{\mathcal{E}}}$ (see \cite[Definition 2.1, 2.2]{CN}).  We let
$$
\Gamma_{\underline{\mathcal{E}}}:{\sf C } \longrightarrow {\sf Gr }B_{\underline{\mathcal{E}}}
$$
denote the functor $\bigoplus_{i\geq 0} \operatorname{Hom}(\mathcal{E}_{-i},-)$, where components of the image in degrees less than $0$ are just $0$ modules.  Recall that, by \cite[Proposition 2.3]{polishproj}, if $\underline{\mathcal{E}}$ is ample, and $X$ is an object of ${\sf C}$, then $\Gamma_{\underline{\mathcal{E}}}(X)$ is a coherent module.

\subsection{The Eilenberg-Watts Theorem for $\mathbb{Z}$-algebras} \label{section.ew}
The goal of this section is to prove a version of the Eilenberg-Watts Theorem for $\mathbb{Z}$-algebras (Proposition \ref{prop:EilenbergWatts}).  This result will be used to compute the automorphism group of noncommutative projective lines (Proposition \ref{prop.autoP1d}), and to classify double covers (Proposition \ref{prop:pullbackFunctor}).  

As the classical Eilenberg-Watts theorem characterizes those functors which are tensoring with a bimodule, we will need to first specify the notion of bimodule we will be using. Below, given a finite dimensional vector space $V$ and object $\mathcal{F} \in {\sf C}$, we let $V \otimes_k \mathcal{F}$ be the object in {\sf C} which represents the covariant functor $V^* \otimes_k \Hom_{\sf C}(\mathcal{F},-) \simeq \Hom_k(V,\Hom_{\sf C}(\mathcal{F},-))$. For $v \in V$ and $\phi \in \Hom_k(V,\Hom_{\sf C}(\mathcal{F},\mathcal{G}))$, we use the notation $\phi(v) \colon v \otimes \mathcal{F} \to \mathcal{G}$. 
\begin{definition} \label{def:AmoduleInC}
An {\em $(A,{\sf C})$-bimodule} is a collection of objects $\mathcal{F}_{\bullet} = \{\mathcal{F}_i\}_{i \in \mathbb{Z}}$ in ${\sf C}$ and morphisms $\mu_{ij} \colon A_{ij} \otimes_k \mathcal{F}_j \to \mathcal{F}_i$ for $i,j \in\mathbb{Z}$ satisfying the usual module axioms.  
\end{definition}
\begin{example}  \label{eg:functorBimodule}
Given a $k$-linear functor $F \colon {\sf coh}(A) \to {\sf C}$ we obtain the $(A,{\sf C})$-bimodule, $F(A)_{\bullet}:= \{F(e_iA)\}$. Indeed, given $a \in A_{ij}$, if $m_a \colon e_jA \to e_iA$ denotes left multiplication by $a$, then $F(a): F(e_j A) \to F(e_i A)$ gives the $A$-module structure. 
\end{example}
Suppose now that $A$ is coherent and $\mathcal{F}_{\bullet}$ is an $(A,{\sf C})$-bimodule. Given $M \in {\sf coh}(A)$ we wish to define $M \otimes_A \mathcal{F}_{\bullet} \in {\sf C}$. If {\sf C} has countably infinite direct sums then we can define 
\begin{equation} \label{eq:AmodInCtensor}
    M \otimes_A \mathcal{F}_{\bullet} := \left( \bigoplus_i M_i \otimes_k \mathcal{F}_i\right)\Big/\Omega \in {\sf C} 
\end{equation}
where $\Omega$ is the sum of the images of the maps 
\begin{equation}  \label{eq:tensorReln}
\rho_{ij}:=m_{ij} \otimes 1 - 1 \otimes \mu_{ij} \colon M_i \otimes A_{ij} \otimes \mathcal{F}_j \to M_j \otimes \mathcal{F}_j \oplus M_i \otimes \mathcal{F}_i
\end{equation}
and $m_{ij}$ denotes module multiplication in $M$. If {\sf C} does not have infinite direct sums, then we need to revert back to the universal property implied by Equations~(\ref{eq:AmodInCtensor}) and (\ref{eq:tensorReln}) as follows. 
\begin{definition}  \label{def:TensorC}
We define $M \otimes_A \mathcal{F}_{\bullet}$ to be the object in {\sf C} that is universal with respect to the following property. 
\begin{enumerate}
    \item There are morphisms $\iota_j \colon M_j \otimes_k \mathcal{F}_j \to M \otimes_A \mathcal{F}_{\bullet}$ for all $j \in \mathbb{Z}$ such that
    \item the composite morphisms
    $$M_i \otimes A_{ij} \otimes \mathcal{F}_j \xrightarrow{\rho_{ij}} M_j \otimes \mathcal{F}_j \oplus M_i \otimes \mathcal{F}_i \xrightarrow{(\iota_j\ -\iota_i)} M \otimes_A \mathcal{F}_{\bullet}$$
    are all zero. 
\end{enumerate}
\end{definition}
\begin{example}  \label{eg:tensoreiA}
One readily checks that $e_i A \otimes_A \mathcal{F}_{\bullet} \simeq \mathcal{F}_i$ with defining morphisms $\iota_j = \mu_{ij} \colon e_i Ae_j \otimes_k \mathcal{F}_j \to \mathcal{F}_i$.
\end{example}
To show existence of $M \otimes_A \mathcal{F}_{\bullet}$ more generally, we give an explicit construction as follows. Note that $M$ is coherent so there is a finite presentation
\begin{equation}  \label{eq:Mpresentation}
    \bigoplus_{s} e_{j_s} A \xrightarrow{\Phi}  \bigoplus_{r} e_{i_r} A \to M \to 0
\end{equation}
where $\Phi$ is viewed as a matrix with entries in $A$. We may thus define
\begin{equation} \label{eq:tensorViaPresentation}
M \underline{\otimes}_A \mathcal{F}_{\bullet} : = \operatorname{coker}\left(\bigoplus_{s} \mathcal{F}_{j_s}\xrightarrow{\Phi}  \bigoplus_{r} \mathcal{F}_{i_r}\right)
\end{equation}
which we shall show is $M \otimes_A \mathcal{F}_{\bullet}$.
\begin{proposition}  \label{prop:constructTensor}
We have $M \underline{\otimes}_A \mathcal{F}_{\bullet} \simeq M \otimes_A \mathcal{F}_{\bullet}$
\end{proposition}
\begin{proof}
We check the universal property for $M \underline{\otimes}_A \mathcal{F}_{\bullet}$. We first construct
$$
\iota_l \colon M_l \otimes_k \mathcal{F}_l \to M \underline{\otimes}_A \mathcal{F}_{\bullet}
$$ 
as follows. Note first that 
\begin{equation*}
    M_l = \frac{\bigoplus_{r} e_{i_r} Ae_l}{\Phi(\bigoplus_{s} e_{j_s} Ae_l)}
\end{equation*}
Now multiplication, as in Example~\ref{eg:tensoreiA} defines a morphism 
\begin{equation*}
  \iota'_l\colon \bigoplus_{r} e_{i_r} Ae_l \otimes \mathcal{F}_l \to \bigoplus_{r} \mathcal{F}_{i_r} \to M \underline{\otimes}_A \mathcal{F}_{\bullet}.
\end{equation*}
By definition of $M \underline{\otimes}_A \mathcal{F}_{\bullet}$, this factors through the quotient $M_l \otimes \mathcal{F}_l$ to give $\iota_l$. The $A$-module compatibility conditions in Definition~\ref{def:TensorC}(2) holds since it holds in $\iota'_l$. 

It remains only to show that $M\underline{\otimes}_A \mathcal{F}_{\bullet}$ is universal so suppose there is another object $\mathcal{G} \in {\sf C}$ equipped with morphisms $\omega_l \colon M_l \otimes_k \mathcal{F}_l \to \mathcal{G}$ satisfying the compatibility condition in Definition~\ref{def:TensorC}(2). The $\omega_{i_r}$ and (\ref{eq:Mpresentation}) can be used to define a morphism $\bigoplus_{r} \mathcal{F}_{i_r} \to \mathcal{G}$ and the compatibility condition ensures it factors through $M \underline{\otimes}_A \mathcal{F}_{\bullet}$ as desired. 
\end{proof}

\begin{proposition}  \label{prop:tensorIsFunctor}
We have a right exact $k$-linear functor $(-) \otimes_A \mathcal{F}_{\bullet} \colon {\sf coh}(A) \to {\sf C}$. 
\end{proposition}
\begin{proof}
Universality defines the $k$-linear functor so we prove right exactness. Given any exact sequence 
\begin{equation*}
   M^* : 0 \to M' \to M \to M'' \to 0
\end{equation*}
of coherent $A$-modules, we can find a Cartan-Eilenberg resolution $P^*_1 \to P^*_0 \to M^* \to 0$. Note that the $P^*_i \otimes_A \mathcal{F}_{\bullet}$  are still split exact. Hence right exactness follows from Proposition~\ref{prop:constructTensor} and right exactness of cokernels. 
\end{proof}
We have our first version of the Eilenberg-Watts theorem.
\begin{proposition}  \label{prop:EilenbergWatts}
Let $F \colon {\sf coh}(A) \to {\sf C}$ be a right exact functor. Then $F$ is naturally isomorphic to $(- ) \otimes_A F(A)_{\bullet}$ where $F(A)_{\bullet}$ is the $(A,{\sf C})$-bimodule in Example~\ref{eg:functorBimodule}. 
\end{proposition}
\begin{proof}
We first construct the morphism $\eta_M \colon M \otimes_A F(A)_{\bullet} \to F(M)$ for any $M \in {\sf coh}(A)$ by using the universal property as follows. We define $\chi_l \colon M_l \otimes F(e_lA) \to F(M)$ to be $m \otimes F(e_lA) \xrightarrow{F(\lambda_m)} F(M)$ for any $m \in M_l$ where $\lambda_m\colon e_lA \to M$ is left multiplication by $m$. Checking the compatibility condition in Definition~\ref{def:TensorC}(2) ensures this induces the morphism $\eta_M$. The naturality of $\eta$ follows from the universal property. To show that $\eta_M$ is an isomorphism, first observe that it is an isomorphism for $M = e_iA$ and then apply the right exact functor $F$ to (\ref{eq:Mpresentation})  to identify $F(M)$ with $M \underline{\otimes}_A F(A)_{\bullet}$. The result now follows from Proposition~\ref{prop:constructTensor}.
\end{proof}
\begin{corollary} \label{cor:EilenbergWatts}
Let $A$ be the $\mathbb{Z}$-algebra associated to some ample sequence so in particular, $A$ is coherent. Let $F \colon {\sf proj}(A) \to {\sf C}$ be a right exact functor. Then $(-)\otimes_A F(\pi A)_{\bullet}$ induces a functor on ${\sf proj}(A)$ and there is a natural isomorphism $\eta \colon (-) \otimes_A F(\pi A)_{\bullet} \to F$. Here $F( \pi A)_{\bullet}$ is the $(A,{\sf C})$-bimodule of Example~\ref{eg:functorBimodule} applied to the functor $F':= F \circ \pi$. 
\end{corollary}
\begin{proof}
Note first that the composite functor $F' \colon {\sf coh}(A) \xrightarrow{\pi} {\sf proj}(A) \xrightarrow{F} {\sf C}$ is right exact so by Proposition~\ref{prop:EilenbergWatts},there is a natural isomorphism $\eta'\colon (-)\otimes_A F'(A)_{\bullet} \to F'$. Since $F'$ maps morphisms with torsion kernel and cokernel to isomorphisms, the same is true of $(-)\otimes_A F'(A)_{\bullet}$ so it follows readily that $\eta'$ induces a natural isomorphism $\eta\colon (-)\otimes_A F(\pi A)_{\bullet} \to F$.
\end{proof}

\subsection{Functors preserving exactness of Euler sequences}
Given objects $\mathcal{F},\mathcal{G} \in {\sf C}$, recall that we have two universal morphisms 
\begin{equation*}
    \iota_{\text{univ}} \colon \mathcal{F} \to \,^*\Hom(\mathcal{F},\mathcal{G}) \otimes \mathcal{G}, \quad \pi_{\text{univ}} \colon \Hom(\mathcal{F},\mathcal{G})  \otimes\mathcal{F} \to \mathcal{G}.
\end{equation*}
Universality for $\iota_{\text{univ}}$ means that given a finite dimensional vector space $V$ and morphism $\iota \colon \mathcal{F} \to V \otimes_k \mathcal{G}$, then there exists a unique linear map $\lambda_{\iota} \colon \ ^*\Hom(\mathcal{F},\mathcal{G}) \to V$ such that $\iota = (\lambda_{\iota} \otimes \operatorname{id}_{\mathcal{G}}) \circ \iota_{\text{univ}}$. Universality for $\pi_{\text{univ}}$ is dual, any morphism $\pi\colon V \otimes \mathcal{F} \to \mathcal{G}$ factors uniquely as $\pi_{\text{univ}} \circ (\rho_{\pi} \otimes \operatorname{id}_{\mathcal{F}})$ for some linear $\rho_{\pi}$. 

\begin{proposition}  \label{prop:functorPreserving Euler}
Suppose $V$ is a finite dimensional vector space, $\mathcal{O}_0,\mathcal{O}_1,\mathcal{O}_2\in {\sf C}$ and 
\begin{equation} \label{eq:EulerTypeSeq}
    0 \to \mathcal{O}_0 \xrightarrow{\iota} V \otimes \mathcal{O}_1 \xrightarrow{\pi} \mathcal{O}_2 \to 0
\end{equation}
is an exact sequence in {\sf C}. 
\begin{enumerate}
    \item Then $\iota$ is determined, up to $\operatorname{Aut}(\mathcal{O}_0)$, by $\pi$, and $\pi$ is determined, up to $\operatorname{Aut}(\mathcal{O}_2)$, by $\iota$. 
    \item Suppose ${\sf D}$ is a $k$-linear, Hom-finite abelian category and $F \colon {\sf C \to D}$ is a $k$-linear functor which preserves exactness of (\ref{eq:EulerTypeSeq}). If $\pi$ is universal (so $V \simeq \Hom(\mathcal{O}_1,\mathcal{O}_2)$) then $F_{\mathcal{O}_1,\mathcal{O}_2}\colon \Hom(\mathcal{O}_1,\mathcal{O}_2) \to \Hom(F\mathcal{O}_1,F\mathcal{O}_2)$ is completely determined, up to $\operatorname{Aut}(F\mathcal{O}_2)$, by the map $F_{\mathcal{O}_0,\mathcal{O}_1} \colon \Hom(\mathcal{O}_0,\mathcal{O}_1) \to \Hom(F\mathcal{O}_0,F\mathcal{O}_1)$. Dually, if $\iota$ is universal, then $F_{\mathcal{O}_0,\mathcal{O}_1}$ is completely determined, up to $\operatorname{Aut}(F\mathcal{O}_0)$, by $F_{\mathcal{O}_1,\mathcal{O}_2}$.
\end{enumerate}
\end{proposition}
\begin{proof}
Part~(1) is clear so we prove (2) and suppose that $\pi$ is universal, the case $\iota$ universal being dual. By assumption we have an exact sequence 
\begin{equation*}
    0 \to F\mathcal{O}_0 \xrightarrow{F\iota} V \otimes F\mathcal{O}_1 \xrightarrow{F\pi} F\mathcal{O}_2 \to 0.
\end{equation*}
First note that we have a commutative diagram
\begin{equation*}
\xymatrix{
V \ar[r]^(.3){\rho_{\pi}} \ar[dr]_{\rho_{F \pi}} &\Hom(\mathcal{O}_1,\mathcal{O}_2) \ar[d]^{F_{\mathcal{O}_1,\mathcal{O}_2}} \\
 & \Hom(F\mathcal{O}_1,F\mathcal{O}_2)}
\end{equation*}
Now $\rho_{\pi}$ is an isomorphism so $F_{\mathcal{O}_1,\mathcal{O}_2}$ is completely determined by $F \pi$, which, by part~(1), is in turn completely, determined, up to $\operatorname{Aut}(F \mathcal{O}_2)$ by $F\iota$. Now $F{\iota}$ is completely determined by $F_{\mathcal{O}_0,\mathcal{O}_1}$. 
\end{proof}

\subsection{The space $\mathbb{P}^{1}_{d}$ and its automorphism group} \label{sub.ncp1}
Before we classify $k$-linear autoequivalences of $\mathbb{P}^{1}_{d}$, it will be convenient to review the construction and basic properties of this space.  Recall that, in \cite{piont}, Piontkovski proves that a $\mathbb{Z}$-graded $k$-algebra 
$$
A=k\langle x_{1}, \ldots, x_{d} \rangle / ( \Sigma_{i=1}^{d}x_{i}\sigma (x_{d-i}) )
$$ 
where $\sigma$ is a graded automorphism of the free algebra, is graded coherent.  Piontkovski defines $\mathbb{P}^{1}_{d}$ as the quotient of the category of graded coherent modules over $A$ modulo the full subcategory of finite-dimensional modules.   

Since $A$ is $\mathbb{Z}$-graded, the category of graded right $A$-modules has a graded shift functor, which descends to a functor $[-]$ in $\mathbb{P}^{1}_{d}$.  By \cite[Section 8]{abstract}, the sequence $\mathcal{L}_{i} := \pi(A [i])$, is ample for $\mathbb{P}^{1}_{d}$, so that there is an equivalence between $\mathbb{P}^{1}_{d}$ and ${\sf proj }\mathbb{S}^{nc}(\underline{\mathcal{L}})$.  Finally, by \cite[Proposition 3.6(1)]{CN}, $\mathbb{S}^{nc}(\underline{\mathcal{L}}) \cong \mathbb{S}^{nc}(\operatorname{Hom}(\mathcal{L}_{-1}, \mathcal{L}_{0}))$, where, for a finite dimensional vector space $V$, $\mathbb{S}^{nc}(V)$ denotes M. Van den Bergh's noncommutative symmetric algebra \cite{vdbp1bundle}.  Since $\operatorname{Hom}(\mathcal{L}_{-1}, \mathcal{L}_{0}) \cong A_{1}$, fixing a basis for $\operatorname{Hom}(\mathcal{L}_{-1}, \mathcal{L}_{0})$ induces an isomorphism $\mathbb{S}^{nc}(\operatorname{Hom}(\mathcal{L}_{-1}, \mathcal{L}_{0})) \longrightarrow \mathbb{S}^{nc}(k^{d})$.  Putting these maps together gives us an equivalence 
$$
{\sf proj }\mathbb{S}^{nc}(k^{d}) \longrightarrow \mathbb{P}^{1}_{d}.
$$
We will often write $\mathbb{S}$ for $\mathbb{S}^{nc}(k^d)$.  

For $i \in \mathbb{Z}$, we let $\mathcal{O}(i)$ denote the object $\pi e_{-i}\mathbb{S} \cong \mathcal{L}_{i}$, and we abuse notation by writing $\mathcal{O}$ for $\mathcal{O}(0)$.  In particular, the sequence $(\mathcal{O}(i))_{i \in \mathbb{Z}}$ is an ample helix for ${\sf proj }\mathbb{S}$. The ``helices'' are the Euler exact sequences
\begin{equation}  \label{eq:EulerSeq}
0 \to \mathcal{O}(i) \to k^d \otimes_k \mathcal{O}(i+1) \to \mathcal{O}(i+2) \to 0
\end{equation}
for $i \in \mathbb{Z}$. The morphisms in this helix are universal so that Proposition~\ref{prop:functorPreserving Euler}(2) applies. 

\begin{proposition}  \label{prop.autoP1d}
The group of $k$-linear auto-equivalences of ${\sf proj}\,\mathbb{S}$ is isomorphic to $\mathbb{Z} \times PGL_d$. The factor $\mathbb{Z}$ is generated by graded shifts whilst $PGL_d$ corresponds to automorphisms of $k^d$.
\end{proposition}
\begin{proof}
It is clear that the automorphism group of ${\sf proj}\,\mathbb{S}$ contains the copy of $\mathbb{Z} \times PGL_d$ described above so it suffices to show that every auto-equivalence $F$ is generated by graded shifts and automorphisms of $k^d$. By Corollary~\ref{cor:EilenbergWatts}, we know that $F$ is naturally isomorphic to $(-) \otimes_{\mathbb{S}} F\mathcal{O}(\bullet)$ where $F\mathcal{O}(\bullet)$ is the $(\mathbb{S},{\sf proj}(\mathbb{S}))$-bimodule $\{F\mathcal{O}(-i)\}$ with scalar multiplication by $\mathbb{S}_{ij} = \Hom(\mathcal{O}(-j),\mathcal{O}(-i))$ defined by 
\begin{equation*}
F_{\mathcal{O}(-j),\mathcal{O}(-i)} \colon \Hom(\mathcal{O}(-j),\mathcal{O}(-i)) \to \Hom(F\mathcal{O}(-j),F\mathcal{O}(-i))
\end{equation*}

Our first goal is to show that $F$ can only shift the sequence $\{ \mathcal{O}(i) | i \in \mathbb{Z}\}$ in the sense that there is some $m \in \mathbb{Z}$ such that $F \mathcal{O}(i) \simeq \mathcal{O}(i+m)$ for all $i \in \mathbb{Z}$. We need to recall some facts from \cite{CNrepStudy}. Let $\sf F$ be the full subcategory of ${\sf proj}\, \mathbb{S}$ of objects isomorphic to direct sums of the $\mathcal{O}(i)$. Then there is a torsion theory $({\sf T, F})$ where ${\sf T}$ consists of torsion sheaves by \cite[Corollary~5.5]{CNrepStudy}. Furthermore, both ${\sf F, T}$ are stable under the Serre functor $\otimes \omega$ and in fact we have $\mathcal{O}(i) \otimes \omega \simeq \mathcal{O}(i-2)$. Thus the $\mathcal{O}(i)$ can be characterised as those indecomposable sheaves $\mathcal{F}$ for which $(\mathcal{F} \otimes \omega^n)_{n \in \mathbb{Z}}$ is an ample sequence. This proves that $F$ must permute the $\mathcal{O}(i)$ and computing Hom spaces, one sees that $F$ can only shift them.

We may thus shift and assume that there are isomorphisms $\eta_i \colon F \mathcal{O}(i) \xrightarrow{\sim} \mathcal{O}(i)$ for all $i$.  Now $F$ induces an automorphism $\Phi$ of $\operatorname{Hom}(\mathcal{O}(-1), \mathcal{O})$, i.e. an element of $GL_{d}$. The Euler exact sequences (\ref{eq:EulerSeq}) mean that we can apply Proposition~\ref{prop:functorPreserving Euler} and conclude that, up to adjusting the $\eta_i$,  $\Phi$ determines all the $F_{\mathcal{O}(-j-1),\mathcal{O}(-j)}$ for all $j$. Now $\mathbb{S}$ is generated in degree 1 so the isomorphism class of the $(\mathbb{S},{\sf proj}(\mathbb{S}))$-bimodule $F\mathcal{O}(\bullet)$ is now completely determined by $\Phi$, so that the result follows from Corollary \ref{cor:EilenbergWatts}.
\end{proof}

\begin{remark} \label{rem.ellipticaut}
Unfortunately, we have not been able to compute the automorphism group of the noncommutative elliptic curve ${\sf C}^{\theta}$. We do know quite a bit because Polischchuk has classified the automorphism group of $D^b(X)$. To understand his result, we recall that the degree and rank map give a surjective group homomorphism $K_0(X) \to \mathbb{Z}^2$. There is thus a group homomorphism $\Psi \colon \operatorname{Aut}D^b(X) \to SL_2(\mathbb{Z})$. By \cite[Theorem~1.2]{polish}, the kernel of this map is generated by automorphisms of the genus one curve $X$ (as an algebraic variety as opposed to a group variety), tensoring by degree 0 line bundles and cohomological shifts in even degree. Automorphisms in this kernel of course also restrict to automorphisms of ${\sf C}^{\theta}$. The image of $\operatorname{Aut}\, {\sf C}^{\theta}$ under $\Psi$ should be the stabilizer $\overline{\operatorname{Aut}} {\sf C}^{\theta}$ in $SL_2(\mathbb{Z})$ of the half-plane where the slope is $> \theta$. Identifying $\overline{\operatorname{Aut}} {\sf C}^{\theta}$ explicitly appears to be a number-theoretic problem we have not solved.
\end{remark}

\subsection{Double covers from two-periodic elliptic helices} Now suppose $\underline{\mathcal{E}}$ is a two-periodic elliptic helix with 
$$
\operatorname{dim }\operatorname{Hom}(\mathcal{E}_{-1},\mathcal{E}_{0}) = d>2.
$$  
Let $\theta$ equal the negative limit slope of $\underline{\mathcal{E}}$.  By \cite[Corollary 3.7]{CN}, $\underline{\mathcal{E}}$ induces a map of $\mathbb{Z}$-algebras
$$
f_{\underline{\mathcal{E}}}: \mathbb{S}^{nc}(\underline{\mathcal{E}}) \rightarrow B_{\underline{\mathcal{E}}},
$$
The functor $f_{\underline{\mathcal{E}}}$ induces a restriction of scalars functor
$$
{f_{\underline{\mathcal{E}}}}_{*}:  {\sf coh }B_{\underline{\mathcal{E}}} \longrightarrow {\sf coh} \mathbb{S}^{nc}(\underline{\mathcal{E}}),
$$
and by \cite[Corollaire 2, p. 368]{gab}, ${f_{\underline{\mathcal{E}}}}_{*}$ descends to a functor
$$
{\sf proj }B_{\underline{\mathcal{E}}} \rightarrow {\sf proj }\mathbb{S}^{nc}(\underline{\mathcal{E}}).
$$
In what follows, we will abuse notation by writing ${f_{\underline{\mathcal{E}}}}_{*}$ for the functor between quotient categories.

On the other hand, since $\underline{\mathcal{E}}$ is ample for ${\sf C}^{\theta}$ by Proposition \ref{prop.smallample}, the composition of the first two functors in
$$
{\sf C}^{\theta} \overset{\Gamma_{\underline{\mathcal{E}}}}{\longrightarrow} {\sf coh }B_{\underline{\mathcal{E}}} \overset{\pi}{\longrightarrow} {\sf proj }B_{\underline{\mathcal{E}}} \overset{{f_{\underline{\mathcal{E}}}}_{*}}{\longrightarrow}  {\sf proj }\mathbb{S}^{nc}(\underline{\mathcal{E}})
$$
is an equivalence.  We denote the triple composition by $F_{\underline{\mathcal{E}}}$, and, as in the proof of \cite[Theorem 1.1]{CN}, $F_{\underline{\mathcal{E}}}$ is a double cover in a suitable sense (see \cite[Corollary 5.8]{CN}).  We further remark that, by \cite[Proposition 3.6(1)]{CN} and Remark \ref{remark.helix}(3), $\mathbb{S}^{nc}(\underline{\mathcal{E}}) \cong \mathbb{S}^{nc}(\operatorname{Hom}(\mathcal{E}_{-1}, \mathcal{E}_{0}))$.  Since fixing a basis for $\operatorname{Hom}(\mathcal{E}_{-1}, \mathcal{E}_{0})$ yields an equivalence ${\sf proj }\mathbb{S}^{nc}(\operatorname{Hom}(\mathcal{E}_{-1}, \mathcal{E}_{0})) \longrightarrow {\sf proj }\mathbb{S}^{nc}(k^{d})$, and since, as we will see below, ${\sf proj }\mathbb{S}^{nc}(k^{d})$ is equivalent to $\mathbb{P}^{1}_{d}$, we will sometimes abuse notation by writing our double cover as
\begin{equation} \label{eqn.maptop1}
F_{\underline{\mathcal{E}}}:{{\sf C}^{\theta} \longrightarrow \mathbb{P}^{1}_{d}}
\end{equation}

\subsection{Comparison of double covers}
In this section, we prove our noncommutative version of \cite[Chapter IV, Lemma 4.4]{hartshorne} classifying $g^1_2$'s. 

We first consider the effect of composing the double cover morphism $F_{\underline{\mathcal{E}}}$ with automorphisms of $\mathbb{P}^1_d$ and ${\sf C}^{\theta}$. This is best approached using the adjoint $f_{\underline{\mathcal{E}}}^*$ to $f_{\underline{\mathcal{E}}*}$ as recalled in the following result.
\begin{proposition}  \label{prop:pullbackFunctor}
The internal tensor functor $- \underline{\otimes}_{\mathbb{S}} B_{\underline{\mathcal{E}}}$ induces a left adjoint 
$$
f_{\underline{\mathcal{E}}}^* \colon {\sf proj}\,\mathbb{S} \to {\sf proj}\, B_{\underline{\mathcal{E}}}
$$ 
to $f_{\underline{\mathcal{E}}*}$. In particular, $F_{\underline{\mathcal{E}}}$ has a left adjoint, say $G_{\underline{\mathcal{E}}}$ and $G_{\underline{\mathcal{E}}} \mathcal{O}(i) \cong \mathcal{E}_i$. The functor $G_{\underline{\mathcal{E}}}$ is determined, up to natural isomorphism by the isomorphism classes of the objects $\mathcal{E}_i$ and the isomorphism $\phi \colon k^d \simeq \operatorname{Hom}(\mathcal{E}_{-1}, \mathcal{E}_{0})$ up to scalar.
\end{proposition}
\begin{proof}
The existence and description of the left adjoint of $f_{\underline{\mathcal{E}}*}$ is \cite[Theorem~4.10]{morph}. Applying the inverse to the natural isomorphism $\pi \circ \Gamma_{\underline{\mathcal{E}}}\colon {\sf C}^{\theta} \to {\sf proj}B_{\underline{\mathcal{E}}}$ to $B_{\underline{\mathcal{E}}}$ shows that $G_{\underline{\mathcal{E}}} \mathcal{O}(i) \cong \mathcal{E}_i$ and in fact, $G_{\underline{\mathcal{E}}}$ is naturally isomorphic to $(-) \otimes_{\mathbb{S}} \underline{\mathcal{E}}$ where $\underline{\mathcal{E}}$ is an $(\mathbb{S},{\sf C}^{\theta})$-bimodule in the obvious way. 
It remains only to show that the isomorphism class of the $(\mathbb{S},{\sf C}^{\theta})$-bimodule $\underline{\mathcal{E}}$ is determined by the isomorphism classes of the $\mathcal{E}_i$ and $\phi$. By (\ref{eq:EulerSeq}) and Remark \ref{remark.helix}(3), Proposition~\ref{prop:functorPreserving Euler} applies to $G_{\underline{\mathcal{E}}}$ so that we can follow the proof of Proposition~\ref{prop.autoP1d}. More precisely, $\phi$ determines the module action of $\mathbb{S}_{01}$ on $\underline{\mathcal{E}}$ and by Proposition~\ref{prop:functorPreserving Euler} and induction, this also determines the module action of all $\mathbb{S}_{i,i+1}$. Since $\mathbb{S}$ is generated in degree 1, we are done. 
\end{proof}

\begin{remark}
It follows from Proposition \ref{prop:pullbackFunctor} that if $\underline{\mathcal{E}}$ and $\underline{\mathcal{E}}'$ are two-periodic elliptic helices which are not in the same numerical class, then the functors $F_{\underline{\mathcal{E}}}$ and $F_{\underline{\mathcal{E}}'}$ are not isomorphic. 
\end{remark}

Proposition \ref{prop:pullbackFunctor} allows us to track what happens to $F_{\underline{\mathcal{E}}}$ when you compose with automorphisms of $\mathbb{P}^1_d$ and ${\sf C}^{\theta}$. We omit the proof of the following result which is a corollary of Proposition~\ref{prop:pullbackFunctor} or more precisely, its proof.
\begin{corollary}  \label{cor:gradedShiftIsHelixShift}
Let $(-)[n]$ be the graded shift by $n$ functor on ${\sf coh}\, \mathbb{P}^1_d$. Then the composite functor $(F_{\underline{\mathcal{E}}}(-))[n]$ is naturally isomorphic to the functor $F_{\underline{\mathcal{E}'}}$ where $\underline{\mathcal{E}'}$ is the shifted helix defined by $\mathcal{E}'_i = \mathcal{E}_{i-n}$.  
\end{corollary}

We now turn to composing with automorphisms of ${\sf C}^{\theta}$. The result of relevance for us is the following, which generalizes \cite[Chapter IV, Lemma 4.4]{hartshorne}.

\begin{proposition}  \label{prop.twistTranslateHelix}
Let $\underline{\mathcal{E}}, \underline{\mathcal{F}}$ be two-periodic helices such that $\deg \mathcal{E}_i = \deg \mathcal{F}_i=: d_i, \ \rk \mathcal{E}_i = \rk \mathcal{F}_i=: r_i$ for $i = -1, 0$ (and hence all $i \in \mathbb{Z}$). Then there is an auto-equivalence $F$ of ${\sf C}^{\theta}$ such that $F_{\underline{\mathcal{E}}} \circ F$ is naturally isomorphic to $F_{\underline{\mathcal{F}}}$ (for some appropriate isomorphism $\operatorname{Hom}(\mathcal{E}_{-1}, \mathcal{E}_0) \simeq \operatorname{Hom}(\mathcal{F}_{-1}, \mathcal{F}_0)$).
\end{proposition}
\begin{proof}
From Proposition~\ref{prop:pullbackFunctor}, it suffices to find an auto-equivalence $G$ of ${\sf C}^{\theta}$ such that $G \mathcal{E}_i \simeq \mathcal{F}_i$ for $i = -1,0$ (and hence by the helix property, all $i \in \mathbb{Z}$). Given that $\mathcal{E}_i$ of rank $r_i$, its isomorphism class is determined completely by the isomorphism class of the determinant bundle $\det \mathcal{E}_i$. Now if $G$ is given by tensoring with the degree 0 line bundle $\mathcal{L}$ then $\det(G \mathcal{E}_i) \simeq \mathcal{L}^{\otimes r_i} \otimes \det(\mathcal{E}_i)$. If $G$ corresponds to pulling back by some translation on the elliptic curve, then $\det(G \mathcal{E}_i) \simeq \mathcal{N}^{\otimes d_i} \otimes \det(\mathcal{E}_i)$ where $\mathcal{N}$ is the degree 0 line bundle corresponding to that translation. It thus suffices to show that for arbitrary degree 0 line bundles $\mathcal{L}_{-1}, \mathcal{L}_0$, we can find $\mathcal{L}, \mathcal{N}$ such that 
\begin{equation*}
    \mathcal{L}^{\otimes r_{-1}} \otimes \mathcal{N}^{\otimes d_{-1}} \simeq \mathcal{L}_{-1}, \ \ \mathcal{L}^{\otimes r_0} \otimes \mathcal{N}^{\otimes d_0} \simeq \mathcal{L}_0.
\end{equation*}
In other words, we need to show that the morphism of abelian varieties $ E \times E \to E \times E$ given by the matrix $\Phi :=\left(\begin{smallmatrix} r_{-1} & d_{-1}\\ r_0 & d_0\end{smallmatrix}\right)$ is surjective. Pre-composing with the adjoint matrix gives the morphism $E \times E \to E \times E$ corresponding to the matrix $\Phi \operatorname{Adj}(\Phi) = \left(\begin{smallmatrix} d & 0\\ 0 & d\end{smallmatrix}\right)$ which is surjective since the multiplication by $d$ map on $E$ is. It follows that $\Phi$ is also surjective as desired. 
\end{proof}

\end{document}